\documentclass[twoside, 12pt]{article}
\usepackage{amsmath,amssymb,amsthm}
\usepackage[margin=2.3cm]{geometry} 

\usepackage{titlesec,hyperref}
\usepackage{fancyhdr}
\usepackage[numbers,sort&compress]{natbib}
\usepackage{color}

\fancypagestyle{plain}
{
\fancyhf{}

}

\titleformat{\subsection}{\it}{\thesubsection.\enspace}{1.5pt}{}
\titleformat{\subsubsection}{\it}{\thesubsubsection.\enspace}{1.5pt}{}

\newtheorem{theo}{Theorem}[section]
\newtheorem{lemm}[theo]{Lemma}

\newtheorem{coro}[theo]{Corollary}
\newtheorem{prop}[theo]{Proposition}
\newtheorem{rema}[theo]{Remark}
\numberwithin{equation}{section}

\allowdisplaybreaks

\def\th2{\frac{\theta}{2}}

\begin{document}
\title{Strong Solutions to the Density-dependent Incompressible Nematic Liquid Crystal Flows  \hspace{-4mm}}
\author{Jincheng Gao$^1$  \quad Qiang Tao $^2$ \quad Zheng-an Yao $^3$\\[10pt]
Department of Mathematics, Sun Yat-Sen University,\\
510275, Guangzhou, P. R. China.\\[5pt]
}
\footnotetext[1]{Email: \it gaojc1998@163.com, Tel:\it +8613450423982}
\footnotetext[2]{Email: \it taoq060@126.com}
\footnotetext[3]{Email: \it mcsyao@mail.sysu.edu.cn}
\date{}

\maketitle

\begin{abstract}
In this paper, we investigate the density-dependent incompressible nematic liquid crystal flows in $n(n=2$ or $3)$ dimensional
bounded domain. More precisely, we obtain the local existence and uniqueness of the solutions when the viscosity coefficient
of fluid depends on density. Moreover, we establish blowup criterions for the regularity of the strong solutions in dimension
two and three respectively. In particular, we build a blowup criterion just in terms of the gradient of density if the initial direction field satisfies some geometric configuration.  For these results, the initial density needs not be strictly positive.

\vspace*{5pt}
\end{abstract}

\section{Introduction}
\quad Nematic liquid crystals contain a large number of elongated, rod-like molecules and possess the same orientational order.
The continuum theory of liquid crystals due to Ericksen and Leslie was developed around $1960$'s \cite{Ericksen, Leslie}
(see also \cite{Gennes}). Since then, numerous researchers have obtained some important developments for liquid crystals
not only in theory but also in the application.  When the fluid containing nematic liquid crystal materials is at rest,
we have the well-known Ossen-Frank theory for static nematic liquid crystals, see the
pioneering work by Hardt-Lin-Kinderlehrer \cite{Hard-Kinderlehrer-Lin} on the analysis of energy minimal configurations of
nematic liquid crystals. Generally speaking, the motion of fluid always takes place. The so-called Ericksen-Leslie system is a
macroscopic descriptions of the time evolution of the materials under the influence of both the flow velocity field and the
macroscopic description of the microscopic orientation configuration of rod-like liquid crystals.
In this paper, we investigate the motion of incompressible nematic liquid crystal flows, which are described by the following simplified version of the Ericksen-Leslie equations:
\begin{equation}\label{1.1}
\left\{
\begin{aligned}
& \rho_t+{\rm div}(\rho u)=0,\\
& (\rho u)_t+{\rm div}(\rho u\otimes u)-{\rm div}(2 \mu(\rho)D(u))+\nabla P=-\lambda {\rm div}(\nabla d \odot \nabla d),\\
& {\rm div}u=0,\\
& d_t+u\cdot \nabla d=\theta (\Delta d+|\nabla d|^2 d),
\end{aligned}
\right.
\end{equation}
in $\Omega \times (0,+\infty),$ where $\Omega$ is a bounded domain with smooth boundary in $\mathbb{R}^n(n=2$ or $ 3 ).$ Here $\rho, u, P$ and $d$ denote the unknown density, velocity, pressure and macroscopic average of the nematic liquid crystal orientation respectively. $D(u)=\frac{\nabla u+\nabla^T u}{2}$ is the deformation tensor, where $\nabla u$ presents the gradient matric of $u$ and $\nabla^T u$ its transpose. $\mu>0, \lambda>0, \theta>0$ are viscosity of fluid, competition between kinetic and potential energy, and microscopic elastic relaxation time respectively. The viscosity coefficient $\mu=\mu(\rho)$ is a general function of density and be assumed to satisfy:
\begin{equation}\label{1.2}
\mu \in C^1[0,\infty) \quad \text{and} \quad \mu>0 \quad \text{on} \quad [0,\infty).
\end{equation}
Without loss of generality, both $\lambda $ and $\theta$ are normalized to $1$.
The symbol $\nabla d \odot \nabla d$ denotes the $n\times n$ matrix whose $(i,j)-$th entry is given by $ \nabla d_i \cdot\nabla d_j,$ for $i,j=1,2,...,n.$ To complete the equation \eqref{1.1}-\eqref{1.2}, we consider an initial boundary value problem
for \eqref{1.1}-\eqref{1.2} with the following initial and boundary conditions
\begin{equation}\label{1.3}
(\rho, u, d)|_{t=0}=(\rho_0, u_0, d_0) \quad \text{with } \quad |d_0|=1, \quad {\rm div}u_0=0 \quad {\rm in} \ \Omega;
\end{equation}
\begin{equation}\label{1.4}
u=0, \quad \frac{\partial d}{\partial \nu}=0 \quad {\rm on} \ \partial \Omega;
\end{equation}
where $\nu$ is the unit outward normal vector to $\partial \Omega$.

When the fluid is the homogeneous case, the systems \eqref{1.1}-\eqref{1.4} are the simplified model of nematic liquid crystals with constant density. When the term $|\nabla d|^2 d$ be replaced by the Ginzburg-Laudan type approximation term $\frac{1-|d|^2}{\varepsilon^2}d$, Lin \cite{Lin} first derived a simplified Ericksen-Leslie
equations modeling the liquid crystal flows in $1989$. Later, Lin and Liu \cite{{Lin-Liu1},{Lin-Liu2}} made some important
analytic studies, such as the existence of weak/strong  solutions and the partial regularity of suitable solutions. Recently,
Dai, Qing and Schonbek \cite{Dai-Qing-Schonbek} studied the large time behavior of solutions and gave the decay rate in the whole space in $\mathbb{R}^3$ with small initial data. Grasselli and Wu \cite{Garsselli-Wu} also considered the long-time behavior
and obtained the estimates on the convergence rate for nematic liquid crystal flows with external force. They also showed the
existence of global strong solutions provided that either the viscosity is large enough or the initial datum is closed to a
given equilibrium. As for the case of  $|\nabla d|^2 d$, Huang and Wang \cite{Huang-Wang} established a blow up criterion
for the short time classical solutions in dimensions two and three. Recently, Li \cite{Li3} proved the local well-posedness
of mild solutions with $L^\infty$ initial data, in particular, the initial energy may be infinite.

When the fluid is non-homogeneous case, we would like to point out that the system \eqref{1.1}-\eqref{1.4} include two important equations, which have attracted large number of analysts' interests:\\
(i) When $d$ is a constant and $\mu$ is a function depending only on the density $\rho$, then system \eqref{1.1}-\eqref{1.4}
    are the well-known Navier-Stokes equations with density-dependent viscosity coefficient. First, Lions \cite{Lions1} established
    the global existence for the weak solutions for the case of positive initial density. As for the uniqueness, Lions telled
    us the fact that sufficiently smooth solutions are unique and any weak solution must be equal to a strong one if the latter
    exists. Later, Cho and Kim \cite{Cho-Kim1} proved the local existence of unique strong solutions for all initial data satisfying a
    natural compatibility condition for the case of initial density being not be strictly positive. He also built the following
    blowup criterion:
\begin{equation}\label{1.5}
\underset{0 \le t \le T^*}{\sup}\left(\|\nabla \rho(t)\|_{L^q}+\|\nabla u(t)\|_{L^2}\right)=\infty,
\end{equation}
if $T^*$ is the maximal existence time of the local strong solutions and $T^*<\infty.$\\
(ii)When $\mu$ is a constant, the systems \eqref{1.1}-\eqref{1.4} are a density-dependent incompressible hydrodynamic flow
    of liquid crystals. When the term $|\nabla d|^2 d$ be replaced by the Ginzburg-Laudan type approximation term
    $\frac{1-|d|^2}{\varepsilon^2}d$, the global existence of weak solution is obtained in \cite{Jiang-Tan, Xu-Tan, Liu-Zhang}
    for each $\varepsilon>0$. Recently, Hu and Wu  \cite{Hu-Wu} proved the decay of the velocity field for arbitrary large
    regular initial data with the initial density being away from vacuum in two dimensional domain with smooth boundary.
    As for the case of  $|\nabla d|^2 d$, Wen and Ding \cite{Wen-Ding} obtained the local existence and uniqueness of strong solutions to the Dirichlet problem in bounded domain with initial density being allowed to have vacuum. Since the strong solutions of a harmonic map can blow up in finite time \cite{Chang-Ding-Ye}, one cannot expect to get a global strong solution with general initial data. Therefore, many researchers attempt to obtain global strong solutions under some additional assumptions.
    Wen and Ding \cite{Wen-Ding} also established the global
    existence and uniqueness of solutions for two dimensional case if the initial density is away from vacuum and the initial
    data is of small norm. Global existence of strong solutions with small initial data to three dimensional liquid crystal
    equations are obtained by Li and Wang in \cite{Li-Wang1} for constant density case, Li and Wang in \cite{Li-Wang2} for
    nonconstant but positive density case, and Ding, Huang and Xia in \cite{Ding-Huang-Xia}. Recently, Li proved the global
    existence and uniqueness  of strong solutions with initial data being of small norm for the dimension two and three in
    bounded domain in \cite{Li1} and the initial direction field satisfying some geometric structure for the two dimensional whole space in \cite{Li2}.

In this paper, we investigate the density-dependent incompressible nematic liquid crystal flows when the viscosity coefficient is a function of the density of the fluid. More precisely, we establish local unique strong solutions to \eqref{1.1}-\eqref{1.4}.
Then, we consider the possible breakdown of regularity for the strong solutions. Firstly, We build up a blowup criterion for the three dimensional bounded domain with smooth boundary. Secondly, applying a logarithmic  inequality, we improve the preceding blowup criterion by omitting the velocity in a two dimensional bounded domain. Lastly, if the initial direction satisfies some geometric configuration, we establish a blowup criterion just in terms of the gradient of the density in two dimensional space.

Before stating our main result, we first explain the notations and conventions used throughout this paper.
We denote
\begin{equation*}
\int fdx=\int_{\Omega}fdx.
\end{equation*}
Let
\begin{equation*}
\dot{f}:=f_t+u \cdot \nabla f
\end{equation*}
represents the material derivative of $f$. For $1\le q \le \infty$ and integer $k\ge 0,$ the standard Sobolev spaces are denoted by:
\begin{equation*}
\left\{
\begin{aligned}
& L^q=L^q(\Omega),\quad W^{k,q}=W^{k,q}(\Omega),\quad H^k=W^{k,2},\\
& W_{0}^{1,q}=\{u\in W^{1,q}|u=0 \ {\rm on}\ \partial \Omega\},\quad H_0^1=W_0^{1,2}.
\end{aligned}
\right.
\end{equation*}
For two $n\times n$ matrices $M=(M_{ij}), N=(N_{ij})$, we denote the scalar product
between $M$ and $N$ by
\begin{equation*}
M:N=\sum_{i,j=1}^{n}M_{ij}N_{ij}.
\end{equation*}
Finally, we recall the definition on the weak $L^p-$space which is defined as follows:
\begin{equation*}
L_w^p \triangleq \left\{ f\in L^1_{loc}: \|f\|_{L^p_w}=
          \underset{t>0}{\sup} \  t \left| \{x \in \Omega: |f(x)>t|\}\right|^{\frac{1}{p}}<\infty \right\}.
\end{equation*}

Now, we state our first result as follows.
\begin{theo}\label{Theorem1.1}
Let $\Omega$ be a bounded smooth domain in $\mathbb{R}^n(n=2,3)$ and $q\in(n,\infty)$ be a fixed constant. Suppose that the
initial data $(\rho_0, u_0, d_0)$ satisfies the regularity conditions
\begin{equation*}
0\le \rho_0 \in W^{1,q}, \ u_0 \in H_0^1\cap H^2, \ d_0 \in H^3 \ {\rm and}\  |d_0|=1,
\end{equation*}
and the compatibility condition
\begin{equation}\label{1.6}
-{\rm div}(2\mu(\rho_0)D(u_0))+\nabla P_0+{\rm div}(\nabla d_0\odot \nabla d_0)=\sqrt{\rho_0}g \ {\rm and } \ {\rm div}u_0=0
\ {\rm in} \ \Omega
\end{equation}
for $(P_0, g)\in H^1 \times L^2$. Then there exist a positive time $T_0>0$ and a unique strong solution $(\rho, u, d, P)$ of
\eqref{1.1}-\eqref{1.4} such that
\begin{equation*}
\begin{aligned}
& 0\le \rho \in C([0,T_0];W^{1,q}), \quad \rho_t \in C([0,T_0]; L^q),\\
& u\in C([0,T_0]; H_0^1 \cap H^2)\cap L^2(0,T_0; W^{2,r}), \\
& u_t \in L^2(0,T_0;H_0^1), \ \sqrt{\rho}u_t \in L^\infty(0,T_0; L^2),\\
& P \in C([0,T_0]; H^1)\cap L^2(0,T_0; W^{1,r}),\\
& d\in C([0,T_0]; H^3)\cap L^2(0,T_0; H^4),\ |d|=1 \ {\rm in} \ \overline{Q_{T_0}},\\
& d_t \in C([0,T_0]; H^1)\cap L^2(0,T_0; H^2), \ d_{tt} \in L^2(0,T_0; L^2),
\end{aligned}
\end{equation*}
for some $r$ with $n < r < min\{q, \frac{2n}{n-2}\}$ and $Q_{T_0}=\Omega \times [0, T_0].$
\end{theo}

After having the Theorem \ref{Theorem1.1} at hand, we will build the following blowup criterion of possible breakdown of local strong solution for the initial boundary problem \eqref{1.1}-\eqref{1.4}  in three dimensional space.

\begin{theo}\label{blowupcriterion1}
Suppose the dimension $n=3$ and all the assumptions in Theorem \ref{Theorem1.1} are satisfied. Let $(\rho, u, d, P)$ be a strong solution of the initial boundary problem \eqref{1.1}-\eqref{1.4}
and $T^*$ be the maximal time of existence.
If $0<T^*<\infty$, then
\begin{equation}\label{1.7}
\underset{T \rightarrow T^*}{\lim}\left(\|\nabla \rho\|_{L^\infty(0,T;L^q)}
    +\|u\|_{L^{s_1}(0,T;L_w^{r_1})}+\|\nabla d\|_{L^{s_2}(0,T;L_w^{r_2})}\right)=\infty,
\end{equation}
where $r_i$ and $s_i$ satisfy
\begin{equation}\label{1.8}
\frac{2}{s_i}+\frac{3}{r_i} \le 1, \ 3<r_i \le \infty, \ i=1,2.
\end{equation}
\end{theo}

As a corollary of Theorem \ref{blowupcriterion1}, we will establish a blowup criterion for the density-dependent incompressible
flow when the viscosity depends on the density in three dimensional domain. More precisely, if $d$ is a constant vector,
then we have the following corollary.
\begin{coro}\label{blowupcriterion2}
Suppose $d$ be a unit constant vector and all the assumptions in Theorem \ref{blowupcriterion1} are satisfied. Let $(\rho, u, d, P)$ be a strong solution of the initial boundary problem \eqref{1.1}-\eqref{1.4}
and $T^*$ be the maximal time of existence.
If $0<T^*<\infty$, then
\begin{equation}\label{1.9}
\underset{T \rightarrow T^*}{\lim}\left(\|\nabla \rho\|_{L^\infty(0,T;L^q)}
    +\|u\|_{L^{s}(0,T;L_w^{r})}\right)=\infty,
\end{equation}
where $r$ and $s$ satisfy
\begin{equation}\label{1.10}
\frac{2}{s}+\frac{3}{r} \le 1, \ 3<r \le \infty.
\end{equation}
\end{coro}

\begin{rema}
The criterion for $u$ in \eqref{1.9} is given by a Serrin type and is more general than the blowup criterion  \eqref{1.5}.
\end{rema}

Our next work is to improve the proceeding blowup criterion \eqref{1.7} by utilizing a logarithmic inequality for the two dimensional space, i.e.,
\begin{theo}\label{blowupcriterion3}
Suppose the dimension $n=2$ and all the assumptions in Theorem \ref{Theorem1.1} are satisfied. Let $(\rho, u, d, P)$ be a strong solution of the initial boundary problem \eqref{1.1}-\eqref{1.4}
and $T^*$ be the maximal time of existence.
If $0<T^*<\infty$, then
\begin{equation}\label{1.11}
\underset{T \rightarrow T^*}{\lim}\left(\|\nabla \rho\|_{L^\infty(0,T;L^q)}+\|\nabla d\|_{L^{s}(0,T;L_w^{r})}\right)=\infty,
\end{equation}
where $r$ and $s$ satisfy
\begin{equation}\label{1.12}
\frac{2}{s}+\frac{2}{r} \le 1, \ 2<r \le \infty.
\end{equation}
\end{theo}

As a corollary of Theorem \ref{blowupcriterion3}, we provide a blowup criterion for the density-dependent incompressible
flow when the viscosity depends on the density in two dimensional domain.
then we have the following corollary.

\begin{coro}\label{blowupcriterion4}
Suppose $d$ be a unit constant vector and all the assumptions in Theorem \ref{blowupcriterion3} are satisfied.
Let $(\rho, u, d, P)$ be a strong solution of the initial boundary problem \eqref{1.1}-\eqref{1.4}
and $T^*$ be the maximal time of existence.
If $0<T^*<\infty$, then
\begin{equation}\label{1.13}
\underset{T \rightarrow T^*}{\lim}\|\nabla \rho\|_{L^\infty(0,T;L^q)}=\infty.
\end{equation}
\end{coro}

Our next work concentrate on building blowup criterion the same as \eqref{1.13} if the initial direction field satisfies some special geometric configuration. More precisely, we have

\begin{coro}\label{blowupcriterion5}
For any $i(i=1,2),$ suppose the $i-$th component of initial direction field $d_{0i}$ satisfies the condition
\begin{equation*}
0 \le \underline{d}_{0i} \le d_{0i} \le 1  \quad {\rm or} \quad -1\le d_{0i} \le -\underline{d}_{0i}\le 0,
\end{equation*}
where $\underline{d}_{0i}$ is defined in \eqref{6.5},
and all the assumptions in Theorem \ref{blowupcriterion3} are satisfied.
Let $(\rho, u, d, P)$ be a strong solution of the initial boundary problem \eqref{1.1}-\eqref{1.4}
and $T^*$ be the maximal time of existence.
If $0<T^*<\infty$, then
\begin{equation}\label{1.14}
\underset{T \rightarrow T^*}{\lim}\|\nabla \rho\|_{L^\infty(0,T;L^q)}=\infty.
\end{equation}
\end{coro}

The rest of the paper is organized as follows: In Section $2$, we present some useful lemmas which will play an important
role in this paper; In Section $3$, we prove the Theorem \ref{Theorem1.1} by applying the method in \cite{Cho-Kim1};
From Section $4$ to Section $6$,  we discuss and verify blowup criterions of strong solution respectively.
\section{Preliminaries}

\quad In this section, we give some useful lemmas which will be used frequently in this paper. The first lemma is the regularity estimates for the stationary Stokes equations, i.e.,

\begin{lemm}\label{Lemma2.1}
(See \cite{Cho-Kim1})Assume $\mu \in C^2[0,\infty)$ and $\rho \in W^{2,q}, 0 \le \rho \le C.$ Let $(u,P)\in H_0^1\times L^2 $ be the unique weak solution to the boundary value problem
\begin{equation*}
-{\rm div}(2\mu(\rho)D(u))+\nabla P=F, \quad {\rm div}u=0 \quad {\rm in}\ \Omega; \ \int Pdx=0,
\end{equation*}
where $D(u)=\frac{\nabla u+\nabla^T u}{2}$. Then we have the following regularity results:\\
$(1)$ If $F\in L^2,$ then $(u,P)\in H^2 \times H^1$ and
\begin{equation}\label{2.1}
\|u\|_{H^2}+\|P\|_{H^1} \le C\|F\|_{L^2}(1+\|\nabla \rho\|_{L^q})^{\frac{q}{q-n}}.
\end{equation}
$(2)$ If $F\in L^r$ for some $r \in (n,q)$ then $(u,P)\in W^{2, r} \times W^{1, r}$ and
\begin{equation}\label{2.2}
\|u\|_{W^{2,r}}+\|P\|_{W^{1,r}} \le C\|F\|_{L^r}(1+\|\nabla \rho\|_{L^q})^{\frac{qr}{2(q-r)}}.
\end{equation}
$(3)$ If $F\in H^1,$ then $(u,P)\in H^3 \times H^2$ and
\begin{equation}\label{2.1}
\|u\|_{H^3}+\|P\|_{H^2} \le \widetilde{C}\|F\|_{H^1}(1+\|\rho\|_{W^{2,q}})^N
\end{equation}
for some $N=N(n,q)>0$. The constant $ \widetilde{C}$ depends also on $\|\partial^2 \mu / \partial \mu^2\|_C$.
\end{lemm}

Next, we introduce a H\"{o}lder inequality in Lorentz space. The Lorentz space and its norm are denoted, respectively, by $L^{p,q}$ and $\|\cdot\|_{L^{p,q}}$, where $1<p<\infty$ and $1 \le q \le \infty.$ Now, we can state the following H\"{o}lder inequality in Lorentz space $L^{p,q}$.

\begin{lemm}\label{Lemma2.2}
(See \cite{Kim} )Let $1<p_1, p_2<\infty$ with $\frac{1}{p}=\frac{1}{p_1}+\frac{1}{p_2}$ and let $1 \le q_1, q_2 \le \infty.$
Then for $f\in L^{p_1, q_1}$ and $g \in L^{p_2,q_2}$, it holds that
\begin{equation*}
\|f \cdot g\|_{L^{p,q}} \le C \| f \|_{L^{p_1, q_1}} \| g \|_{L^{p_1, q_1}} \quad {\rm with} \ q=\min\{q_1, q_2 \},
\end{equation*}
where $C$ is a positive constant depending only on $p_1, p_2, q_1$ and $q_2$.
\end{lemm}

The following lemma has been proved in \cite{Kim}, for the readers' convenience, we give the proof in detail.

\begin{lemm}\label{lemma2.3}
Assume $g\in H^1$ and $f\in L_w^r$ with $r\in (n, \infty],$ then $f\cdot g \in L^2.$ Furthermore, for any $\varepsilon>0$ and
$r \in (n, \infty]$, we have
\begin{equation}\label{2.4}
\|f\cdot g\|_{L^2}^2 \le \varepsilon \|g\|_{H^1}^2+C(\varepsilon)\|f\|_{L_w^r}^{\frac{2r}{r-n}}\|g\|_{L^2}^2,
\end{equation}
where $C(\varepsilon)$ is  a positive constant depending only on $\varepsilon, \ n, \ r$ and the domain $\Omega$.
\end{lemm}
\begin{proof}
Applying the Lemma \ref{Lemma2.2}, it is easy to get
\begin{equation}\label{2.5}
\|f \cdot g\|_{L^2}=\|f \cdot g\|_{L^{2,2}}\le C\|f\|_{L^r_w}\|g\|_{L^{\frac{2r}{r-2},2}}, \quad {\rm where}\ r\in (n, \infty].
\end{equation}
Next, we will show that
\begin{equation}\label{2.6}
\|g\|_{L^{\frac{2r}{r-2},2}} \le C \|g\|_{L^2}^{\frac{r-n}{r}} \|g\|_{H^1}^{\frac{n}{r}}, \quad {\rm where} \ r\in (n, \infty].
\end{equation}
Indeed, if $r=\infty$, then \eqref{2.6} holds on obviously. If $r\in (n,\infty)$, then $L^{\frac{2r}{r-2},2}$ is a real interpolation space of $ L^{\frac{2 r_1}{r_1-2}}$ and
$L^{\frac{2 r_2}{r_2-2}}$, where $r_1, r_2$ and $r$ satisfy $n<r_1<r<r_2<\infty$ and $\frac{2}{r}=\frac{1}{r_1}+\frac{1}{r_2}$, then
\begin{equation*}
\begin{aligned}
\|g\|_{L^{{\frac{2r}{r-2}},2}}
& \le C\|g\|_{L^{\frac{2 r_1}{r_1-2}}}^{\frac{1}{2}}\|g\|_{L^{\frac{2 r_2}{r_2-2}}}^{\frac{1}{2}}\\
& \le C\left(\|g\|_{L^2}^{\frac{r_1-n}{r_1}}\|g\|_{H^1}^{\frac{n}{r_1}}\right)^{\frac{1}{2}}
       \left(\|g\|_{L^2}^{\frac{r_2-n}{r_2}}\|g\|_{H^1}^{\frac{n}{r_2}}\right)^{\frac{1}{2}} \\
& \le C\|g\|_{L^2}^{\frac{r-n}{r}}\|g\|_{H^1}^{\frac{n}{r}},
\end{aligned}
\end{equation*}
where we have used the Sobolev inequality. Therefore, combining \eqref{2.5} with \eqref{2.6} gives \eqref{2.4} directly if we exploit the Cauchy inequality.
\end{proof}

The last lemma introduced in this section will be the following logarithmic Sobolev inequality which plays an important role
in the proof of the Lemma \ref{lemma5.2}. Omitting the proof for brief, one can read \cite{Huang-Wang2}.
\begin{lemm}\label{lemma2.4}
Let $\Omega$ be a bounded smooth domain in $\mathbb{R}^2$, and $f\in L^2(s,t;H^1\cap W^{1,q})$ for $q\in (2,\infty).$ Then
there exists a constant $C$ depending only on $q$ such that
\begin{equation}\label{2.7}
\|f\|^2_{L^2(s,t;L^\infty)}\le C\left[1+\|f\|^2_{L^2(s,t;H^1)}\ln (e+\|f\|_{L^2(s,t;W^{1,q})})\right],
\end{equation}
where $C$ depends only on $q$ and $\Omega$, but independent of $s, t$.
\end{lemm}

\section{Proof of Theorem \ref{Theorem1.1}}

\quad In this section, we will only give the existence proof for the Theorem \ref{Theorem1.1} since the uniqueness of the solution is easy to obtain by a standard argument(c.f.\cite{Lions1}). In order to solve the initial boundary problem \eqref{1.1}-\eqref{1.4}, we will split the proof into three parts. In part one, we will establish the global strong solution for some linearized systems. In part two, we prove the solution of the linearized systems converges to the original initial problem \eqref{1.1}-\eqref{1.4} in a local time  for positive initial density. In part three, we verify Theorem \ref{Theorem1.1} for the case of general initial density.

\subsection{Global existence for the linearized equations}

\quad We consider the following linearized systems:
\begin{equation}\label{3.1}
\left\{
\begin{aligned}
&\rho_t+v \cdot \nabla \rho=0,\\
&\rho u_t +\rho v \cdot \nabla u-{\rm div}(2\mu(\rho)D(u))+\nabla P=-{\rm div}f,\\
&{\rm div}u=0,
\end{aligned}
\right.
\end{equation}
in $(0,\infty)\times \Omega$, where\ $2D(u)=\nabla u+\nabla^T u, \mu=\mu(\rho)$ satisfies \eqref{1.2}  and $v$ is a known divergence-free vector field.
Then we will state the main result in this subsection.

\begin{prop}\label{Proposition3.1}
Assume that the data $(\rho_0, u_0, f)$ satisfies the regularity conditions:
\begin{equation*}
0 \le \rho_0 \in W^{1,q}, \ u_0 \in H_0^1 \cap H^2 \ {\rm and} \ f\in L^\infty(0,T; H^1)\cap L^2(0,T; H^2), f_t \in L^2(0,T; H^1)
\end{equation*}
for some $q$ with $n<q<\infty$, and the compatibility condition
\begin{equation}\label{3.2}
-{\rm div}(2\mu(\rho_0)D(u_0))+\nabla P_0+{\rm div}f_0=\sqrt{\rho_0}g \quad  {\rm and} \ {\rm div}u_0=0 \ {\rm in} \ \Omega,
\end{equation}
for some $(P_0, g)\in H^1 \times L^2$. If in addition, $v$ satisfies the regularity conditions
\begin{equation*}
v\in L^\infty(0,T; H_0^1 \cap H^2)\cap L^2(0,T; W^{2,r}),
\ v_t \in L^2(0,T; H_0^1) \ {\rm and} \ {\rm div}v=0 \ {\rm in} \ \Omega,
\end{equation*}
for some $r$ with $n<r<\min \{q, \frac{2n}{n-2}\}$. Then there exists a unique strong solution $(\rho, u, P)$ to the initial boundary value problem \eqref{3.1}, \eqref{1.2}-\eqref{1.4} such that
\begin{equation}\label{3.3}
\begin{aligned}
& \rho \in C([0,T]; W^{1,q}), \ \nabla u, \ P \in C([0,T]; H^1)\cap L^2(0,T; W^{1,r}),\\
& \rho_t \in C([0,T]; L^q),\ \sqrt{\rho}u_t \in L^\infty(0,T; L^2), \ u_t \in L^2 (0,T; H_0^1).
\end{aligned}
\end{equation}
\end{prop}

In order to prove Proposition \ref{Proposition3.1}, we will take by three steps:\\
$(1)$ In addition to the assumptions in Proposition \ref{Proposition3.1}, if suppose
\begin{equation}\label{3.4}
\mu \in C^2 [0, \infty), \quad \rho_0 \in W^{2,q}, \quad \rho_0 \ge \delta \ {\rm for\ some} \ \delta >0,
\end{equation}
then we give the prove of Proposition \ref{Proposition3.1};\\
$(2)$ To remove the additional condition \eqref{3.4}, we need to derive some uniform estimates independent of
 $\delta, \|\rho_0\|_{W^{2,q}}$ and $ \|\partial^2 \mu / \partial \mu^2\|_{C}$;\\
$(3)$ Having the results of the proceeding two steps at hand, it is a standard argument to give the proof of
Proposition \ref{Proposition3.1}.

Now, let us to begin our first step. Indeed, taking the method in \cite{Cho-Kim1}, it is easy to get the following results.
For a brief, we only state the results and omit the proof.
\begin{lemm}\label{lemma3.2}
In addition to the hypotheses of Proposition \ref{Proposition3.1}, we assume that the condition \eqref{3.4} are satisfied. Then there exists a unique strong solution $(\rho, u, P)$ to the initial boundary value problem  \eqref{3.1}, \eqref{1.2}-\eqref{1.4} such that
\begin{equation*}
\begin{aligned}
& \rho \in C(0,T; W^{1, \infty}) \cap L^2(0,T; W^{2,r}),\ \rho_t \in L^\infty(0,T; W^{1,r}),\\
& u \in C([0,T]; H_0^1 \cap H^2)\cap L^2(0,T; H^3),\ u_t \in L^\infty(0,T; L^2)\cap L^2(0,T; H_0^1),\\
& P\in L^\infty(0, T; H^1)\cap L^2(0,T; H^2),
\end{aligned}
\end{equation*}
where $n<r<\min \{q, \frac{2n}{n-2}\}$.
\end{lemm}

Thanks to the previous Lemma \ref{lemma3.2}, there exists a unique strong solution $(\rho, u, P)$ satisfying the
regularity \eqref{3.3}. To remove the additional hypotheses \eqref{3.4}, we will derive some uniform estimates independent of
$\delta, \|\rho_0\|_{W^{2,q}}$ and $ \|\partial^2 \mu / \partial \mu^2\|_{C}$.

\begin{lemm}\label{lemma3.3}
Suppose $(\rho, u, P)$ be the strong solution to the problem  \eqref{3.1}, \eqref{1.2}-\eqref{1.4}, then we have
\begin{equation}\label{3.5}
\begin{aligned}
&\underset{0 \le t \le T}{\sup}\left(\|\rho\|_{W^{1,q}}+\|\rho_t\|_{L^q}+\|u\|_{H^2}+\|P\|_{H^1}+\|\sqrt{\rho}u_t\|_{L^2}\right)\\
&+\int_0^T \left(\|u\|_{W^{2,r}}^2+\|P\|_{W^{1,r}}^2+\|\nabla u_t\|_{L^2}^2\right)dt \le C,
\end{aligned}
\end{equation}
where $C$ independent of $\delta, \|\rho_0\|_{W^{2,q}}$ and $ \|\partial^2 \mu / \partial \mu^2\|_{C}$
and  $n<r<\min \{q, \frac{2n}{n-2}\}$.
\end{lemm}

\begin{proof}
\textrm{Step 1:} We deduce from \eqref{3.1}$_1$ by applying the characteristic method that
\begin{equation}\label{3.6}
\|\rho(t)\|_{L^s}=\|\rho_0\|_{L^s} \quad {\rm for} \quad 0 \le t \le T, \ 1\le s \le \infty.
\end{equation}
Taking the gradient operator to \eqref{3.1}$_1$, multiplying by $q|\nabla \rho|^{q-2}\nabla \rho$ and integrating by parts, we obtain
\begin{equation*}
\|\nabla \rho(t)\|_{L^q}\le \|\nabla \rho_0\|_{L^q}\exp \left(C\int_0^t \|v(s)\|_{W^{2,r}}ds\right),
\end{equation*}
which implies
\begin{equation*}
\|\partial_t \rho(t)\|_{L^q}\le C\|v\|_{H^2}\|\nabla \rho_0\|_{L^q}\exp \left(C\int_0^t \|v(s)\|_{W^{2,r}}ds\right),
\end{equation*}
due to  \eqref{3.1}$_1$.
It is easy to observe from \eqref{1.2} and \eqref{3.6} that
\begin{equation}\label{3.7}
C^{-1} \le \mu \le C \quad {\rm and} \quad |\nabla \mu|\le C|\nabla \rho|,
\end{equation}
which will be used repeatedly.

\textrm{Step 2:} Multiplying $\eqref{1.1}_2$ by $u_t$ and integrating over $(0,t)\times\Omega$, we have
\begin{equation}\label{3.8}
\begin{aligned}
&\quad \int \mu (\rho)|D(u)|^2(t) dx+\int_0^t \int \rho |u_t|^2 dxd\tau\\
&\le C+\int |f||\nabla u|dx
 +\int_0^t \int \left(|f_t||\nabla u|+\rho |v| |\nabla u| |u_t|+|\mu'||v||\nabla \rho||D(u)|^2\right)dxd\tau\\
&=C+\sum_{i=1}^4 I_{1i}.
\end{aligned}
\end{equation}
To estimate the terms $I_{1i}(1\le i \le 4)$, we will make use of the Young and Sobolev inequalities.
\begin{equation*}
\begin{aligned}
I_{11} & \le C(\varepsilon)\int |f|^2dx+\varepsilon\int |\nabla u|^2dx,\\
I_{12} & \le 2 \int_0^t(\|f_t\|_{L^2}^2+\|\nabla u\|_{L^2}^2)d\tau,\\
I_{13} & \le C(\varepsilon)\int_0^t\int \rho |v|^2 |\nabla u|^2 dxd\tau+\varepsilon\int_0^t \int \rho |u_t|^2 dxd\tau\\
       & \le C(\varepsilon)\|\rho\|_{L^\infty} \|v\|_{L^\infty}^2 \int_0^t \int |\nabla u|^2 dxd\tau
           +\varepsilon\int_0^t \int \rho |u_t|^2 dxd\tau,\\
I_{14} & \le C\int_0^t \|v\|_{L^6} \|\nabla \rho\|_{L^q} \|\nabla u\|_{L^{\frac{12q}{5q-6}}}^2d\tau\\
       & \le C\int_0^t \|\nabla v\|_{L^2} \|\nabla \rho\|_{L^q}
           \|\nabla u\|_{L^2}^{\frac{12q-n(q+6)}{6q}}\|\nabla u\|_{H^1}^{\frac{n(q+6)}{6q}}d\tau\\
       & \le C \|\nabla \rho\|_{L^q}  \|\nabla v\|_{L^2} \int_0^t \|\nabla u\|_{L^2}^2d\tau
           +\varepsilon \int_0^t \|\nabla u\|_{H^1}^2 d\tau.
\end{aligned}
\end{equation*}
On the other hand, we get
\begin{equation}\label{3.9}
\int \mu(\rho) |D(u)|^2 dx \ge C^{-1}\int |D(u)|^2 dx=\frac{1}{2C}\int |\nabla u|^2 dx,
\end{equation}
due to \eqref{3.7} and $\eqref{3.1}_3$.
Substituting $I_{1i}(1 \le i\le 4)$ and \eqref{3.9} into \eqref{3.8} yields
\begin{equation}\label{3.10}
\begin{aligned}
&\quad \frac{1}{4C}\int |\nabla u|^2dx+\frac{1}{2}\int_0^t\int \rho |u_t|^2 dxd\tau\\
& \le C\left(1+\|f\|_{L^2}^2+\int_0^t \|f_t\|_{L^2}^2 d\tau\right)+C\int_0^t \|\nabla u\|_{L^2}^2 d\tau
    +\varepsilon \int_0^t \|\nabla u\|_{H^1}^2 d\tau.
\end{aligned}
\end{equation}
In order to deal with the term $\int_0^t \|\nabla u\|_{H^1}^2 d\tau$, we will applying the regularity estimate
for the stationary Stokes equations in Lemma \ref{Lemma2.1}. More precisely,
\begin{equation*}
\|u\|_{H^2}+\|P\|_{H^1}\le C \|F\|_{L^2}(1+\|\nabla \rho\|_{L^q})^{\frac{q}{q-n}}\le C\|F\|_{L^2},
\end{equation*}
where
\begin{equation*}
\begin{aligned}
\|F\|_{L^2}& =\|-\rho u_t -\rho v \cdot \nabla u-{\rm div}f\|_{L^2}\\
           &\le C\|\sqrt{\rho}u_t\|_{L^2}+C\|\nabla u\|_{L^2}+\|{\rm div}f\|_{L^2}.
\end{aligned}
\end{equation*}
Then we have the following regularity estimate
\begin{equation}\label{3.11}
\|u\|_{H^2}+\|P\|_{H^1}\le C (\|{\rm div}f\|_{L^2}+\|\sqrt{\rho}u_t\|_{L^2}+\|\nabla u\|_{L^2}).
\end{equation}
Substituting \eqref{3.11} into \eqref{3.10} and choosing $\varepsilon$ small enough, we obtain
\begin{equation*}
\int |\nabla u|^2 dx+\int_0^t \int \rho |u_t|^2 \le C+C\int_0^t \|\nabla u\|_{L^2}^2d\tau,
\end{equation*}
which, if we exploit the Gronwall inequality, implies
\begin{equation}\label{3.12}
\underset{0 \le t \le T}{\sup}\|\nabla u\|_{L^2}^2 +\int_0^t \|\sqrt{\rho}u_t\|_{L^2}^2 d\tau\le C.
\end{equation}

\textrm{Step 3:}Differentiating $\eqref{1.1}_2$ with respect to $t$, multiplying by $u_t$ and integrating over $\Omega$, we have
\begin{equation}\label{3.13}
\begin{aligned}
& \quad  \frac{1}{2} \frac{d}{dt}\int \rho |u_t|^2dx+\frac{1}{C}\int |\nabla u_t|^2 dx\\
& \le C\int \left(|v| |\nabla \rho| |\nabla u| |\nabla u_t|+\rho |v| |u_t| |\nabla u_t|
                 + |v|^2 |\nabla \rho| |\nabla u| |u_t|\right.\\
&\left. \quad \quad \quad \quad +\rho |v_t| |\nabla u| |u_t|+|f_t||\nabla u_t|\right)dx
=\sum_{i=1}^5 I_{2i}.
\end{aligned}
\end{equation}
To estimate the terms $I_{2i}(1\le i \le 5)$, we will apply \eqref{3.11}, the Gagliardo-Nirenberg and H\"{o}lder inequalities repeatedly:
\begin{equation*}
\begin{aligned}
I_{21} &\le \|v\|_{L^\infty}\|\nabla \rho\|_{L^q}\|\nabla u\|_{L^{\frac{2q}{q-2}}}\|\nabla u_t\|_{L^2}\\
       &\le C(\varepsilon)\|v\|_{L^\infty}^2 \|\nabla \rho\|_{L^q}^2 \|\nabla u\|_{H^1}^2+\varepsilon\|\nabla u_t\|_{L^2}^2\\
       &\le C(\varepsilon)(1+\|{\rm div}f\|_{L^2}^2+\|\sqrt{\rho}u_t\|_{L^2}^2)+\varepsilon\|\nabla u_t\|_{L^2}^2,\\
I_{22} &\le C(\varepsilon)\|\rho\|_{L^\infty}\|v\|_{L^\infty}^2 \int \rho |u_t|^2dx+\varepsilon \int |\nabla u_t|^2dx\\
       &\le C(\varepsilon)\|\sqrt{\rho}u_t\|_{L^2}^2+\varepsilon\|\nabla u_t\|_{L^2}^2,\\
I_{23} &\le \|v\|_{L^\infty}^2\|\nabla \rho\|_{L^q}\|\nabla u\|_{L^{\frac{6q}{5q-6}}}\|u_t\|_{L^6}\\
       &\le C\|v\|_{L^\infty}^2\|\nabla \rho\|_{L^q}
            \|\nabla u\|_{L^2}^{\frac{3q+nq-3n}{6q}}\|\nabla u\|_{H^1}^{\frac{3q+3n-nq}{6q}}\|\nabla u_t\|_{L^2}\\
       &\le C(\varepsilon)(1+\|{\rm div}f\|_{L^2}^2+\|\sqrt{\rho}u_t\|_{L^2}^2)+\varepsilon\|\nabla u_t\|_{L^2}^2,\\
I_{24} &\le \|\rho\|_{L^\infty}\|v_t\|_{L^3}\|\nabla u\|_{L^2}\|u_t\|_{L^6}\\
       &\le C\|\rho\|_{L^\infty}\|\nabla v_t\|_{L^2}\|\nabla u\|_{L^2}\|\nabla u_t\|_{L^2}\\
       &\le C(\varepsilon)\|\nabla v_t\|_{L^2}^2++\varepsilon\|\nabla u_t\|_{L^2}^2,\\
I_{25} &\le C(\varepsilon)\|f_t\|_{L^2}^2+\varepsilon\|\nabla u_t\|_{L^2}^2.\\
\end{aligned}
\end{equation*}
For any fixed $\tau\in(0, t)$, substituting $I_{2i}(1\le i \le 5)$ into \eqref{3.13} and integrating over $(\tau, t)\subset[0, T]$ yield
\begin{equation*}
\begin{aligned}
&\quad \frac{1}{2}\int\rho |u_t|^2 dx+\frac{1}{2C}\int_\tau^t \int |\nabla u_t|^2 dxds \\
&\le \frac{1}{2}\int \rho(\tau)|u_t(\tau)|^2 dx
     +C\int_\tau^t\left(\|{\rm div}f\|_{L^2}^2+\|f_t\|_{L^2}^2+\|\nabla v_t\|_{L^2}^2\right)ds+\int_\tau^t\int \rho |u_t|^2dxds.
\end{aligned}
\end{equation*}
Thanks to the compatibility condition \eqref{3.2}, letting $\tau \rightarrow 0^+$ and applying the Gr\"{o}nwall inequality, it arrives at
\begin{equation}\label{3.14}
\underset{0\le t \le T}{\sup}\|\sqrt{\rho}u_t(t)\|_{L^2}^2+\int_0^T \|\nabla u_t\|_{L^2}^2dt \le C.
\end{equation}

\textrm{Step 4:} High order estimates. Indeed, combing \eqref{3.14} with \eqref{3.11}-\eqref{3.12} yields
\begin{equation}\label{3.15}
\|u\|_{H^2}+\|P\|_{H^1}\le C.
\end{equation}
Applying the stationary Stokes regularity estimate, i.e. \eqref{2.2}, we get
\begin{equation*}
\|u\|_{W^{2,r}}+\|P\|_{W^{1,r}}\le C\|F\|_{L^r}(1+\|\nabla \rho\|_{L^q})^{\frac{qr}{2(q-r)}}\le C\|F\|_{L^r},
\end{equation*}
where
\begin{equation*}
\begin{aligned}
\|F\|_{L^r}&=\|-\rho u_t-\rho v \cdot \nabla u -{\rm div}f\|_{L^r}\\
           &\le \|\rho\|_{L^\infty}\|u_t\|_{L^r}+\|\rho\|_{L^\infty}\|v\|_{L^\infty}\|\nabla u\|_{L^r}+\|{\rm div}f\|_{L^r}\\
           &\le C\left(\|\rho\|_{L^\infty}\|\nabla u_t\|_{L^2}+\|\rho\|_{L^\infty}\|v\|_{L^\infty}\|\nabla u\|_{H^1}
                        +\|{\rm div}f\|_{H^1}\right)\\
           &\le C(1+\|\nabla u_t\|_{L^2}+\|{\rm div}f\|_{H^1}).
\end{aligned}
\end{equation*}
Hence we obtain the following regularity estimate
\begin{equation}\label{3.16}
\|u\|_{W^{2,r}}+\|P\|_{W^{1,r}}\le C(1+\|\nabla u_t\|_{L^2}+\|{\rm div}f\|_{H^1}).
\end{equation}
Therefore, we complete the proof of lemma.
\end{proof}

After having the Lemmas \ref{lemma3.2}-\ref{lemma3.3} at hand, we turn to prove the Proposition \ref{Proposition3.1}.
We only sketch the proof here since it is a standard argument(c.f.\cite{Cho-Kim1}). Let $(\rho_0, u_0)$ be an initial data
satisfying the hypotheses of Proposition \ref{Proposition3.1}. For each $\delta \in(0,1)$, choose $\rho_0^\delta \in W^{2,q}$ and $\mu^\delta \in C^2[0, \infty)$ such that
\begin{equation*}
0<\delta \le \rho_0^\delta \le \rho_0+1, \quad \rho_0^\delta \rightarrow \rho_0 \ {\rm in} \ W^{1, q} \
{\rm and} \ \mu^\delta \rightarrow \mu \ {\rm in} \ C^1[0, \infty),
\end{equation*}
as $\delta \rightarrow 0$, and denote $(u_0^\delta, P_0^\delta)\in H_0^1 \times L^2$ a solution to the problem
\begin{equation*}
-{\rm div}(\mu^\delta(\rho_0^\delta)D(u_0^\delta))+\nabla P_0^\delta+{\rm div}f_0=\sqrt{\rho_0^\delta}g \quad {\rm and} \quad
{\rm div}u_0^\delta=0 \quad {\rm in} \ \Omega.
\end{equation*}
Then, applying the Lemma \ref{lemma3.3}, the corresponding solution $(\rho^\delta, u^\delta, P^\delta)$ satisfies the estimate
\begin{equation*}
\begin{aligned}
&\underset{0 \le t \le T}{\sup}\left(\|\rho^\delta\|_{W^{1,q}}+\|\rho^\delta_t\|_{L^q}+\|u^\delta\|_{H^2}+\|P^\delta\|_{H^1}
   +\|\sqrt{\rho^\delta}u^\delta_t\|_{L^2}\right)\\
&\quad \quad \quad +\int_0^T \left(\|u^\delta\|_{W^{2,r}}^2+\|P^\delta\|_{W^{1,r}}^2+\|\nabla u^\delta_t\|_{L^2}^2\right)dt \le C,
\end{aligned}
\end{equation*}
where $C$ independent of $\delta, \|\rho_0\|_{W^{2,q}}$ and $ \|\partial^2 \mu / \partial \mu^2\|_{C}$
and  $n<r<\min \{q, \frac{2n}{n-2}\}$.
We choose a subsequence of solutions $(\rho^\delta, u^\delta)$ which converge to a limit $(\rho, u)$ in a weak sense.
Therefore, it is a strong solution to the linearized problem satisfying the regularity estimates in Lemma \ref{lemma3.3}. Thus, we complete the proof of Proposition \ref{Proposition3.1}.

\subsection{Local existence for the original problem}

\quad In this section, we will prove a local existence result on strong solution with positive initial density to the original problem \eqref{1.1}-\eqref{1.4} at first. Furthermore, we derive some uniform bounds which are independent of the lower bounds of the initial density. Then, these uniform bounds will be used to prove the existence of strong solution with nonnegative initial
density in the last part of this section.

\begin{prop}\label{proposition3.4}
Assume that the data $(\rho_0,u_0,d_0)$ satisfies the regularity conditions
\begin{equation*}
\rho_0 \in W^{1,q}, \ u_0 \in H_0^1 \cap H^2, \ d_0 \in H^3,
\end{equation*}
for some $q$ with $n<q<\infty$ and the compatibility condition
\begin{equation}\label{3.17}
-{\rm div}(2 \mu(\rho_0)D(u_0))+\nabla P_0+{\rm div}(\nabla d_0 \odot \nabla d_0)=\sqrt{\rho_0}g
   \quad {\rm and} \quad {\rm div}u_0=0 \quad {\rm in}\ \Omega,
\end{equation}
for $(P_0, g)\in H^1\times L^2.$ Assume further that $\rho_0 \ge \delta$ in $\Omega$ for some constant $\delta>0.$
Then there exist a time $T_0\in(0, T)$ and a unique strong solution $(\rho, u, P, d)$ to the nonlinear problem
\eqref{1.1}-\eqref{1.4} such that
\begin{equation}\label{3.18}
\begin{aligned}
&\rho \in C([0, T_0]; W^{1, q}), \quad \rho_t \in C([0, T_0]; L^q),\\
&u\in C([0, T_0]; H_0^1 \cap H^2)\cap L^2(0, T_0; W^{2,r}),\\
&u_t \in L^2(0,T_0; H_0^1),\  \sqrt{\rho}u_t\in L^\infty(0, T_0; L^2),\\
&P\in C([0, T_0]; H^1)\cap L^2(0, T_0; W^{1,r}),\\
&d \in C([0, T_0]; H^3)\cap L^2(0, T_0; H^4),\ |d|=1 \ {\rm in}\ \overline{Q_{T_0}},\\
&d_t \in C([0, T_0]; H^1)\cap L^2(0, T_0; H^2), \ d_{tt}\in L^2(0, T_0; L^2),
\end{aligned}
\end{equation}
for some $r$ with $n<r<\min\{q, \frac{2n}{n-2}\}$.
\end{prop}
To prove the Proposition \ref{proposition3.4}, we first construct approximate solutions, as follows:\\
$(1)$ first define $u^0=0$ and $d^0=d_0$;\\
$(2)$ assuming that $u^{k-1}$ and $d^{k-1}$ was defined for $k\ge 1$, let $(\rho^k, u^k, d^k, P^k)$ be the unique
solution to the following initial boundary value problem
\begin{equation}\label{3.19}
\left\{
\begin{aligned}
&\rho_t^k+u^{k-1}\cdot \nabla \rho^k=0,\\
&\rho^k u^k_t+\rho^k (u^{k-1}\cdot \nabla)u^k- {\rm div}(2\mu(\rho^k)D(u^k))+\nabla P^k
=-{\rm div}(\nabla d^k \odot \nabla d^k ),\\
&{\rm div}u^k=0,\\
&d^k_t-\Delta d^k=|\nabla d^{k-1}|^2 d^{k-1}-(u^{k-1}\cdot \nabla)d^{k-1},
\end{aligned}
\right.
\end{equation}
with the initial and boundary conditions
\begin{equation}\label{3.20}
\left.(\rho^k, u^k, d^k)\right|_{t=0}=(\rho_0, u_0, d_0) \quad x\in \Omega,
\end{equation}
\begin{equation}\label{3.21}
(u^k, \frac{\partial d^k}{\partial \nu})=(0,0) \quad {\rm on}\ \partial \Omega,
\end{equation}
where $\nu$ is the unit outward normal vector to $\partial \Omega$.

\subsubsection{Uniform bounds}
\quad Thanks to the Proposition \ref{Proposition3.1} to $\eqref{3.19}_1-\eqref{3.19}_2$ and existence and uniqueness of the theory of parabolic equation to $\eqref{3.19}_4 $(c.f.\cite{Ladyzenskaja-Solonnikov-Ural'ceva}), it is easy to get the existence of a global strong solution $(\rho^k, u^k, P^k, d^k)$
with the regularity \eqref{3.18} to the linearized problem \eqref{3.19}-\eqref{3.21}.

From now on, we derive uniform bounds on the approximate solutions and then prove that the approximate solutions converge
to a strong solution of the original nonlinear problem. Let $K\ge 1$ be a fixed large integer, and define a function as
\begin{equation*}
\Phi_K(t)=\underset{1\le k\le K}{\max}\underset{0 \le s \le t}{\sup}
          \left(1+\|\nabla u^k(s)\|_{L^2}+\|\nabla d^k(s)\|_{H^2}+\|\nabla \rho^k(s)\|_{L^q}\right).
\end{equation*}
Observe then that
\begin{equation}\label{3.22}
\delta \le \rho^k \le C, \quad C^{-1} \le \mu^k \le C, \quad |\nabla \mu^k| \le C|\nabla \rho^k|.
\end{equation}
Then we will estimate each term of $\Phi_K(t)$ in terms of some integral of $\Phi_K(t)$.

\begin{lemm}\label{lemma3.5}
There exists a positive constant $N=N(n,q)$ such that
\begin{equation}\label{3.23}
\|\nabla u^k(t)\|_{L^2}^2+\int_0^t \|\sqrt{\rho^k_t}u^k (s)\|_{L^2}^2ds \le C+C\int_0^t \Phi_K(s)^N ds
\end{equation}
for all $k,\ 1\le k \le K.$
\end{lemm}
\begin{proof}
Multiplying $\eqref{3.19}_2$ by $u_t^k$, integrating by parts and making use of $\eqref{3.20}_1$, we have
\begin{equation*}
\begin{aligned}
&\quad \int \rho^k |u_t^k|^2 dx+\frac{d}{dt}\int \mu(\rho^k)|D(u^k)|^2dx\\
& =-\int \rho^k (u^{k-1}\cdot \nabla) u^k \cdot u_t^kdx- \int \mu' (u^{k-1} \cdot \nabla \rho^k) |D(u^k)|^2dx
  +\int \nabla d^k \odot \nabla d^k : \nabla u_t^k dx,
\end{aligned}
\end{equation*}
On account of the identity
\begin{equation*}
\int \nabla d^k \odot \nabla d^k : \nabla u_t^k dx
=\frac{d}{dt}\int \nabla d^k \odot \nabla d^k : \nabla u^k dx
-\int \nabla d_t^k \odot \nabla d^k : \nabla u^k+\nabla d_t^k \odot \nabla d_t^k : \nabla u^k dx.
\end{equation*}
We integrate over $(0, t)$ and apply \eqref{3.9} and $\eqref{3.19}_3$ to deduce that
\begin{equation}\label{3.24}
\begin{aligned}
& \quad \frac{1}{2C}\int |\nabla u^k|^2 dx-\int |\nabla d^k|^2|\nabla u^k|dx+\frac{1}{2}\int_0^t\int \rho^k |u_t^k|^2dxd\tau\\
& \le C+C\int_0^t \int
   \left(\rho^k |u^{k-1}|^2|\nabla u^k|^2+|u^{k-1}| |\nabla \rho^k| |\nabla u^k|+|\nabla d^k| |\nabla d_t^k| |\nabla u^k|\right)dxd\tau\\
&=C+\sum_{i=1}^3 I_{3i}.
\end{aligned}
\end{equation}
Applying the Gagliardo-Nirenberg, H\"{o}lder and Young inequalities repeatedly, it arrives at
\begin{equation*}
\begin{aligned}
I_{31}  & \le C\int_0^t \|u^{k-1}\|_{L^6}^2 \|\nabla u^k\|_{L^3}^2 d\tau
          \le C\int_0^t \|\nabla u^{k-1}\|_{L^2}^2 \|\nabla u^k\|_{L^2}^{\frac{6-n}{3}}\|\nabla u^k\|_{H^1}^{\frac{n}{3}}d\tau\\
        & \le\varepsilon \int_0^t \|\sqrt{\rho^k}u_t^k\|_{L^2}^2d\tau+C(\varepsilon)\int_0^t \Phi_K^{N_2}d\tau,\\
I_{32} & \le C\int_0^t \|\nabla \rho^k\|_{L^q} \|\nabla u^{k-1}\|_{L^2}\|\nabla u^k\|_{L^{\frac{12q}{5q-6}}}^2 d\tau\\
       & \le C\int_0^t \|\nabla \rho^k\|_{L^q} \|\nabla u^{k-1}\|_{L^2}
              \|\nabla u^k\|_{L^2}^{\frac{12q-n(q+6)}{6q}}\|\nabla u^k\|_{H^1}^{\frac{n(q+6)}{6q}}d\tau\\
      &\le  \varepsilon \int_0^t \|\sqrt{\rho^k}u_t^k\|_{L^2}^2d\tau+C(\varepsilon)\int_0^t \Phi_K^{N_3}d\tau,\\
I_{33} &\le  \int_0^t \|\nabla d^k\|_{L^6}\|\nabla d_t^k\|_{L^2}\|\nabla u^k\|_{L^3}d\tau
        \le  C\int \|\nabla d^k\|_{H^1} \|\nabla d_t^k\|_{L^2}
         \|\nabla u^k\|_{L^2}^{\frac{6-n}{3}} \|\nabla u^k\|_{H^1}^{\frac{n}{3}}d\tau\\
     & \le\varepsilon \int_0^t \|\sqrt{\rho^k}u_t^k\|_{L^2}^2d\tau+C(\varepsilon)\int_0^t \Phi_K^{N_4}d\tau\\
\end{aligned}
\end{equation*}
for some $N_i=N_i(n,q)>0(i=2,3,4),$\ where we have used the following regularity estimate
\begin{equation}\label{3.25}
\|u^k\|_{H^2}+\|P^k\|_{H^1}\le C(1+\|\sqrt{\rho^k}u_t^k\|_{L^2})\Phi_K^{N_1} \quad {\rm for \ some}\ N_1=N_1(n,q)>0.
\end{equation}
Substituting $I_{3i}(i=1,2,3)$ into \eqref{3.24} and choosing $\varepsilon$ small enough yield
\begin{equation}\label{3.26}
\begin{aligned}
\frac{1}{2C}\int |\nabla u^k|^2 dx-\int |\nabla d^k|^2 |\nabla u^k|dx+\frac{1}{4}\int_0^t\int \rho^k |u_t^k|^2dxd\tau
\le C+C\int_0^t \Phi^{N_{5}}_K d\tau
\end{aligned}
\end{equation}
for some $N_5=N_5(n, q)>0$.
In order to control the term $-\int |\nabla d^k|^2 |\nabla u^k|dx$ on left hand side of \eqref{3.26}, taking $\nabla$ operator to $\eqref{3.19}_4$, then we have
\begin{equation}\label{3.27}
\nabla d_t^k -\nabla \Delta d^k=\nabla \left[|\nabla d^{k-1}|^2 d^{k-1}-(u^{k-1}\cdot \nabla)d^{k-1}\right].
\end{equation}
Multiplying by $4|\nabla d^k|^2 \nabla d^k$, integrating (by parts)over $\Omega$ and exploiting the boundary condition
$\left.\frac{\partial d^k}{\partial \nu}\right|_{\partial \Omega}=0$, we get
\begin{equation}\label{3.28}
\begin{aligned}
&\quad \frac{d}{dt}\int |\nabla d^k|^4 dx+4\int |\nabla d^k|^2 |\Delta d^k|^2  dx\\
&\le C\int \left(|\nabla d^k|^3 |\nabla u^{k-1}| |\nabla d^{k-1}|+|\nabla d^k|^3 |u^{k-1}| |\nabla^2 d^{k-1}|
            +|\nabla d^k|^3 |\nabla d^{k-1}|^3\right.\\
&\quad \quad \quad \quad \left.+|\nabla d^k|^3 |\nabla d^{k-1}| |\nabla^2 d^{k-1}|+|\nabla d^k|^2|\nabla^2 d^k|^2\right)dx
 =\sum_{i=1}^5 I_{4i},
\end{aligned}
\end{equation}
where we have used the fact
\begin{equation*}
\begin{aligned}
-\int \nabla \Delta d^k \cdot 4|\nabla d^k|^2 \nabla d^k dxc
&=\int 4|\nabla d^k|^2 |\Delta d^k|^2 dx+8\int \partial_j \partial_j d^k \partial_l \partial_i d^k \partial_l d^k \partial_i d^k dx.
\end{aligned}
\end{equation*}
Applying the H\"{o}lder and Gagliardo-Nirenberg inequalities, we have
\begin{equation*}
\begin{aligned}
I_{41} &\le \|\nabla d^{k-1}\|_{L^\infty} \|\nabla d^k\|_{L^6}^3 \|\nabla u^{k-1}\|_{L^2}
        \le \|\nabla d^{k-1}\|_{H^2} \|\nabla d^k\|_{H^1}^3  \|\nabla u^{k-1}\|_{L^2}\le C\Phi_K^5,\\
I_{42} &\le \|\nabla d^k\|_{L^6}^3 \|u^{k-1}\|_{L^6} \|\nabla^2 d^{k-1}\|_{L^3}
         \le  \|\nabla d^k\|_{H^1}^3 \|\nabla u^{k-1}\|_{L^2} \|\nabla^2 d^{k-1}\|_{H^1}\le C\Phi_K^5,\\
I_{43} &\le \|\nabla d^k\|_{L^6}^3 \|\nabla d^{k-1}\|_{L^6}^3\le C\|\nabla d^k\|_{H^1}^3\|\nabla d^{k-2}\|_{H^1}^3
         \le C\Phi_K^6,\\
I_{44} &\le \|\nabla d^k\|_{L^6}^3 \|\nabla d^{k-1}\|_{L^6}\|\nabla^2 d^{k-1}\|_{L^3}
         \le \|\nabla d^k\|_{H^1}^3\|\nabla d^{k-1}\|_{H^1}\|\nabla^2 d^{k-1}\|_{H^1}\le C\Phi_K^5.\\
I_{45} &\le \|\nabla d^k\|_{L^\infty}^2 \|\nabla^2 d^k\|_{L^2}^2
         \le C\|\nabla d^k\|_{H^2}^2 \|\nabla^2 d^k\|_{L^2}^2 \le C\Phi_K^4.
\end{aligned}
\end{equation*}
Substituting $I_{4i}(i=1,2,3,4,5)$ into \eqref{3.28} and integrating over $(0,t)$, it arrives at
\begin{equation}\label{3.29}
\int |\nabla d^k|^4dx+4\int_0^t\int |\nabla d^k|^2|\Delta d^k|^2dxd\tau \le C+C\int_0^t \Phi_K^6(s)ds.
\end{equation}
Choosing a constant $C_{*}$ sufficiently large such that
\begin{equation*}
\frac{1}{2C}|\nabla u^k|^2-|\nabla u^k| |\nabla d^k|^2+C_{*} |\nabla d^k|^4
\ge \frac{1}{4C}|\nabla u^k|^2+\frac{C_{*}}{2}|\nabla d^k|^4,
\end{equation*}
then $\eqref{3.26}+\eqref{3.29}\times C_{*}$ yields
\begin{equation*}
\int |\nabla u^k|^2dx+\int_0^t \int \rho^k |u_t^k|^2dxd\tau \le C+C\int_0^t \Phi_K^{N_6} d\tau,
\end{equation*}
for some $N_6=N_6(n,q)>0$. Hence, we complete the proof of lemma.
\end{proof}

Next, we estimate the term $\|\sqrt{\rho^k}u_t^k\|_{L^2}$ and $\|\nabla u_t^k\|_{L^2}$ to guarantee the higher regularity.

\begin{lemm}\label{lemma3.6}
There exists a positive constant $N=N(n,q)$ such that
\begin{equation}\label{3.30}
\|\sqrt{\rho^k}u_t^k(t)\|_{L^2}^2+\int_0^t \|\nabla u_t^k\|_{L^2}^2 ds \le C\exp\left[C\int_0^t \Phi_K^N(s)ds\right]
\end{equation}
for any $k,\ 1\le k \le K.$
\end{lemm}
\begin{proof}
Differentiating $\eqref{3.19}_2$ with respect to $t$, multiplying by $u_t^k$ and using $\eqref{3.19}_1$, then we get
\begin{equation}\label{3.31}
\begin{aligned}
&\quad \frac{1}{2}\frac{d}{dt}\int \rho^k |u_t^k|^2dx+\frac{1}{C}\int |\nabla u_t^k|^2 dx\\
&\le C \int \left(\rho^k|u^{k-1}| |u_t^k| |\nabla u_t^k| +\rho^k |u^{k-1}| |\nabla u^{k-1}| |u_t^k|
           +\rho^k |u^{k-1}|^2 |\nabla^2 u^k| |u_t^k|
           \right.\\
&\quad \quad \quad \left. + \rho^k |u^{k-1}|^2 |\nabla u^k| |\nabla u_t^k|+\rho^k |u_t^{k-1}| |\nabla u^k| |\nabla u_t^k|
           +|\nabla d^k| |\nabla d_t^k| |\nabla u_t^k|\right. \\
&\quad \quad \quad   \left. +|\partial_t \rho^k| |\nabla u^k| |\nabla u_t^k|\right)dx
=\sum_{i=1}^7 I_{5i}.
\end{aligned}
\end{equation}
To estimate the term $I_{5i}(1\le i \le 7)$, we make use of the H\"{o}lder, Gagliardo-Nirenberg and Young inequalities.
\begin{equation*}
\begin{aligned}
I_{51} & \le \|\rho^k\|_{L^\infty}^{\frac{1}{2}} \|u^{k-1}\|_{L^6}\|\sqrt{\rho^k}u_t^k\|_{L^3}\|\nabla u_t^k\|_{L^2}\\
       &  \le C\|\rho^k\|_{L^\infty}^{\frac{1}{2}} \|\nabla u^{k-1}\|_{L^2}
             \|\sqrt{\rho^k}u_t^k\|_{L^2}^{\frac{1}{2}}\|\sqrt{\rho^k}u_t^k\|_{L^6}^{\frac{1}{2}}\|\nabla u_t^k\|_{L^2}\\
       &\le C(\varepsilon)\|\sqrt{\rho^k}u_t^k\|_{L^2}^2 \Phi_K^4 +\varepsilon\|\nabla u_t^k\|_{L^2}^2,\\
I_{52} &\le \|\rho^k\|_{L^\infty}\|u^{k-1}\|_{L^6}\|\nabla u^{k-1}\|_{L^2}\|\nabla u^k\|_{L^6}\|u_t^k\|_{L^6}\\
       & \le C(\varepsilon)\|\nabla u^{k-1}\|_{L^2}^4 \|u^k\|_{H^2}^2+\varepsilon\|\nabla u_t^k\|_{L^2}^2\\
       &\le C(\varepsilon)(1+\|\sqrt{\rho^k}u_t^k\|_{L^2}^2) \Phi_K^{4+2N_1}+\varepsilon\|\nabla u_t^k\|_{L^2}^2,\\
I_{53} &\le \|\rho^k\|_{L^\infty} \|\nabla u^{k-1}\|_{L^2}^2 \|\nabla^2 u^k\|_{L^2} \|u_t^k\|_{L^6}\\
       & \le C(\varepsilon)\|\nabla u^{k-1}\|_{L^2}^4 \|u^k\|_{H^2}^2+2+\varepsilon\|\nabla u_t^k\|_{L^2}^2\\
       &\le C(\varepsilon)(1+\|\sqrt{\rho^k}u_t^k\|_{L^2}^2) \Phi_K^{4+2N_1}+\varepsilon\|\nabla u_t^k\|_{L^2}^2,\\
I_{54} &\le \|\rho^k\|_{L^\infty} \|\nabla u^{k-1}\|_{L^2}^2 \|\nabla u^k\|_{H^1} \|\nabla u_t^k\|_{L^2}\\
       & \le C(\varepsilon)\|\nabla u^{k-1}\|_{L^2}^4 \|u^k\|_{H^2}^2+2+\varepsilon\|\nabla u_t^k\|_{L^2}^2\\
       &\le C(\varepsilon)(1+\|\sqrt{\rho^k}u_t^k\|_{L^2}^2) \Phi_K^{4+2N_1}+\varepsilon\|\nabla u_t^k\|_{L^2}^2,\\
I_{55} &\le \|\rho^k\|_{L^\infty}^{\frac{1}{2}}  \|u_t^{k-1}\|_{L^6} \|\nabla u^k\|_{L^2} \|\sqrt{\rho^k} u_t^k\|_{L^3}\\
       & \le C\|\nabla u_t^{k-1}\|_{L^2} \|\nabla u^k\|_{L^2}
             \|\sqrt{\rho^k} u_t^k\|_{L^2}^{\frac{1}{2}}\|\sqrt{\rho^k} u_t^k\|_{L^2}^{\frac{1}{2}}\\
       &\le C(\varepsilon,\eta)\|\sqrt{\rho^k}u_t^k\|_{L^2}^2 \Phi_K^4 +\varepsilon\|\nabla u_t^k\|_{L^2}^2
               +\eta\|\nabla u_t^{k-1}\|_{L^2}^2,\\
I_{56} &\le C(\varepsilon)\int |\nabla d^k|^2 |\nabla d^k_t|^2 dx+\varepsilon\|\nabla u_t^k\|_{L^2}^2\\
       & \le C(\varepsilon)\|\nabla d^k\|_{L^\infty}^2\|\nabla d_t^k\|_{L^2}^2+\varepsilon\|\nabla u_t^k\|_{L^2}^2\\
       &\le C(\varepsilon)\|\nabla d^k\|_{H^2}^2\|\nabla d_t^k\|_{L^2}^2+\varepsilon\|\nabla u_t^k\|_{L^2}^2,
\end{aligned}
\end{equation*}
where we have used \eqref{3.25}.
In order to control the term $\|\nabla d_t^k\|_{L^2}$, applying the $L^2-$estimate to the \eqref{3.27}, we obtain
\begin{equation}\label{3.32}
\begin{aligned}
&\quad \|\nabla d_t^k\|_{L^2}\le \|\nabla\left(\Delta d^k+|\nabla d^{k-1}|^2 d^{k-1}-u^{k-1}\cdot \nabla d^{k-1}\right)\|_{L^2}\\
&\le C\left(\|\nabla^3 d^k\|_{L^2}+\|\nabla d^{k-1}\|_{L^6}\|\nabla^2 d^{k-1}\|_{L^3}+\|\nabla d^{k-1}\|_{L^6}^3
            \right.\\
&\quad \quad \quad \left.+\|\nabla d^{k-1}\|_{L^\infty}\|\nabla u^{k-1}\|_{L^2}+\|u^{k-1}\|_{L^6}\|\nabla^2 d^{k-1}\|_{L^3}\right)\\
&\le  C\left(\|\nabla^3 d^k\|_{L^2}+\|\nabla d^{k-1}\|_{H^1}\|\nabla^2 d^{k-1}\|_{H^1}+\|\nabla d^{k-1}\|_{H^1}^3
            \right.\\
&\quad \quad \quad \left.+\|\nabla d^{k-1}\|_{H^2}\|\nabla u^{k-1}\|_{L^2}+\|u^{k-1}\|_{H^1}\|\nabla^2 d^{k-1}\|_{H^1}\right)\\
& \le C\Phi_K^3.
\end{aligned}
\end{equation}
Substituting  \eqref{3.32} into $I_{56} $ to deduce that
\begin{equation*}
I_{56} \le C(\varepsilon)\Phi_K^8+\varepsilon\|\nabla u_t^k\|_{L^2}^2.
\end{equation*}
Since the term $I_{57}$ is somewhat complicated, we will deal with it as follows:
If $n=2$, then
\begin{equation*}
\begin{aligned}
I_{57} &\le \int |u^{k-1}| |\nabla \rho^k| |\nabla u^k| |\nabla u_t^k| dx\\
     &\le \|\nabla \rho^k\|_{L^q}\left\| |u^{k-1}| |\nabla u^k| \right\|_{L^{\frac{2q}{q-2}}}\|\nabla u_t^k \|_{L^2}\\
    &\le\|\nabla \rho^k\|_{L^q} \|\nabla u^{k-1}\|_{L^2} \|\nabla u^k\|_{H^1}\|\nabla u_t^k \|_{L^2}\\
    &\le C(\varepsilon)(1+\|\sqrt{\rho^k}u_t^k\|_{L^2}^2) \Phi_K^{4+2N_1}+\varepsilon\|\nabla u_t^k\|_{L^2}^2,\\
\end{aligned}
\end{equation*}
If $n=3$, then
\begin{equation*}
\begin{aligned}
I_{57} &\le \|\nabla \rho^k\|_{L^q}\|u^{k-1}\|_{L^6} \|\nabla u^k\|_{L^{\frac{3q}{q-3}}} \|\nabla u_t^k\|_{L^2}\\
    &\le \|\nabla \rho^k\|_{L^q}\|\nabla u^{k-1}\|_{L^2}
      \|\nabla u^k\|_{L^2}^{\frac{2q-6}{3q}} \|\nabla u^k\|_{L^\infty}^{\frac{q+6}{3q}}\|\nabla u_t^k\|_{L^2}\\
    &\le C(\varepsilon)\Phi_K^{N_8}+\varepsilon\|\nabla u_t^k\|_{L^2}^2,
\end{aligned}
\end{equation*}
where we have used the regularity estimate
\begin{equation}\label{3.33}
\|u^k\|_{W^{2,r}}+\|P^k\|_{W^{1,r}}\le C(1+\|\nabla u_t^k\|_{L^2})\Phi_K^{N_7} \quad {\rm for \ some}\ N_7=N_7(n,q)>0.
\end{equation}
Substituting $I_{5i}(1\le i\le 7)$ into \eqref{3.31} and choosing $\varepsilon$ small enough, we get
\begin{equation*}
\frac{1}{2}\frac{d}{dt}\int \rho^k |u_t^k|^2 dx+\frac{1}{2C}\int |\nabla u_t^k|^2 dx
\le (1+\|\sqrt{\rho^k}u_t^k\|_{L^2}^2) \Phi_K^{N_9} +\eta\|\nabla u_t^{k-1}\|_{L^2}^2.
\end{equation*}
Fixing $\tau $ in $(0, T)$ and integrating over $(\tau, t)\subset(0, T)$, we have
\begin{equation*}
\begin{aligned}
&\quad \|\sqrt{\rho^k}u_t^k(t)\|_{L^2}^2+\int_\tau^t \|\nabla u_t^k\|_{L^2}^2 ds\\
&\le C\|\sqrt{\rho^k}u_t^k(\tau)\|_{L^2}^2+C\int_\tau^t(1+\|\sqrt{\rho^k}u_t^k\|_{L^2}^2)\Phi_K^{N_9}ds
+\frac{1}{2}\int_\tau^t \|\nabla u_t ^{k-1}\|_{L^2}^2 ds.
\end{aligned}
\end{equation*}
From the recursive relation of $\|\nabla u_t^k\|_{L^2}$, we get
\begin{equation*}
\begin{aligned}
\int_\tau^t \|\nabla u_t^k\|_{L^2}^2 ds
&\le C \left(\sum_{i=1}^{k}\frac{1}{2^{i-1}}\right)
          \left(\|\sqrt{\rho^k}u_t^k(\tau)\|_{L^2}^2
                   +\int_\tau^t (1+\|\sqrt{\rho^k}u_t^k\|_{L^2}^2)\Phi_K^{N_9}ds\right)\\
&\le 2C\left(\|\sqrt{\rho^k}u_t^k(\tau)\|_{L^2}^2
                   +\int_\tau^t (1+\|\sqrt{\rho^k}u_t^k\|_{L^2}^2)\Phi_K^{N_9}ds\right).
\end{aligned}
\end{equation*}
Hence, we have the following estimate
\begin{equation*}
\|\sqrt{\rho^k}u_t^k(t)\|_{L^2}^2+\int_\tau^t \|\nabla u_t^k\|_{L^2}^2 ds
\le 2C\left(\|\sqrt{\rho^k}u_t^k(\tau)\|_{L^2}^2
                   +\int_\tau^t (1+\|\sqrt{\rho^k}u_t^k\|_{L^2}^2)\Phi_K^{N_9}ds\right).
\end{equation*}
Thanks to the compatibility condition \eqref{3.17}, making use of the Gr\"{o}nwall  inequality, it arrives at
\begin{equation*}
\|\sqrt{\rho^k}u_t^k(t)\|_{L^2}^2+\int_0^t \|\nabla u_t^k\|_{L^2}^2 ds \le C\exp\left[C\int_0^t \Phi_K^{N_9}(s)ds\right],
\end{equation*}
which completes the proof of the lemma.
\end{proof}

As a corollary of Lemma \ref{lemma3.6},  we can obtain the following estimate immediately.

\begin{lemm}\label{lemma3.7}
There exists a positive constant $N=N(n,q)$ such that
\begin{equation}\label{3.34}
\|\nabla \rho^k(t)\|_{L^q}\le C \exp\left[C\exp(C\int_0^t \Phi_K^N(s)ds)\right]
\end{equation}
for any $k,\ 1\le k \le K.$
\end{lemm}

Now we turn to the estimate the second term $\|\nabla d^k\|_{H^2}$ in $\Phi_K$. Indeed, we should obtain the
following estimate first.

\begin{lemm}\label{lemma3.8}
There exists a positive constant $N=N(n,q)$ such that
\begin{equation}\label{3.35}
\|\nabla d^k(t)\|_{H^1}^2+\int_0^t \|\nabla d_t^k\|_{L^2}^2 ds \le C +C\int_0^t \Phi_K^N ds
\end{equation}
for any $k,\ 1\le k \le K.$
\end{lemm}
\begin{proof}
\textrm{Step 1:} Multiplying $\eqref{3.19}_4$ by $\Delta d^k$ and integrating (by parts) over $\Omega$, we get
\begin{equation*}
\quad \frac{1}{2}\frac{d}{dt}\int |\nabla d^k|^2 dx+\int |\Delta d^k|^2 dx\\
=\int \left[(u^{k-1} \cdot \nabla)d^{k-1}-|\nabla d^{k-1}|^2d^{k-1}\right]\cdot \Delta d^k dx.
\end{equation*}
Applying the H\"{o}lder and Gagliardo-Nirenberg inequalities, it arrives at
\begin{equation*}
\begin{aligned}
&\quad \frac{1}{2}\frac{d}{dt}\int |\nabla d^k|^2 dx+\int |\Delta d^k|^2 dx\\
&\le \int (|u^{k-1}| |\nabla d^{k-1}| |\nabla^2 d^k|+ |\nabla d^{k-1}|^2 |d^{k-1}| |\nabla^2 d^k|)dx\\
&\le \|u^{k-1}\|_{L^6} \|\nabla d^{k-1}\|_{L^3} \|\nabla^2 d^k\|_{L^2}+
     \|\nabla d^{k-1}\|_{L^3}^2 \|\nabla^2 d^{k-1}\|_{L^3}
\le C\Phi_K^3.
\end{aligned}
\end{equation*}
Integrating the proceeding inequality over $(0, t)$, we have
\begin{equation}\label{3.36}
\int |\nabla d^k|^2 dx+\int_0^t \int |\Delta d^k|^2 dxd\tau \le C+C\int_0^t \Phi_K^3d\tau.
\end{equation}

\textrm{Step 2:}Multiplying \eqref{3.27} by $\nabla d_t^k$ and integrating (by parts) over $\Omega$ yield
\begin{equation*}
\begin{aligned}
&\quad \frac{1}{2}\frac{d}{dt}\int |\Delta d^k|^2 dx+\int |\nabla d_t^k|^2 dx\\
&\le C(\varepsilon)\int \left(|\nabla u^{k-1}|^2 |\nabla d^{k-1}|^2+|u^{k-1}|^2 |\nabla^2 d^{k-1}|^2
                   +|\nabla d^{k-1}|^6+|\nabla d^{k-1}|^2 |\nabla^2 d^{k-1}|^2\right)dx\\
&\quad \quad +\varepsilon\int |\nabla d_t^k|^2 dx,
\end{aligned}
\end{equation*}
where we have used the Cauchy inequality.
Choosing $\varepsilon$ small enough, integrating over $(0, t)$ and applying the H\"{o}lder, Sobolev inequality, it arrives at
\begin{equation*}
\begin{aligned}
&\quad \int |\Delta d^k(t)|^2 dx+\int_0^t \int |\nabla d_t^k|^2dxds\\
&\le C+C\int_0^t \int \left( |\nabla u^{k-1}|^2 |\nabla d^{k-1}|^2+|u^{k-1}|^2 |\nabla^2 d^{k-1}|^2
                   +|\nabla d^{k-1}|^6+|\nabla d^{k-1}|^2 |\nabla^2 d^{k-1}|^2\right)dxds\\
&\le C+C\int_0^t \left(\|\nabla d^{k-1}\|_{H^2}^2 \|\nabla u^{k-1}\|_{L^2}^2
        +\|\nabla u^{k-1}\|_{L^2}^2\|\nabla^2 d^{k-1}\|_{H^1}^2
        +\|\nabla d^{k-1}\|_{H^1}^6\right.\\
&\quad \left.+\|\nabla d^{k-1}\|_{H^2}^2\|\nabla^2 d^{k-1}\|_{L^2}^2\right)ds\\
&\le C+C\int_0^t \Phi_K^6 ds.
\end{aligned}
\end{equation*}
Hence, if applying the elliptic regularity, we obtain
\begin{equation*}
\int |\nabla^2 d^k(t)|^2 dx+\int_0^t \int |\nabla d_t^k|^2dxds\\
\le C+C\int_0^t \Phi_K^6 ds,
\end{equation*}
which, together with \eqref{3.36}, complete the proof of this lemma.
\end{proof}

Now, we can derive the high order estimate for $d^k$.
\begin{lemm}\label{lemma3.9}
There exists a positive constant $N=N(n,q)$ such that
\begin{equation}\label{3.37}
\|\nabla^3 d^k(t)\|_{L^2}^2+\int_0^t \|\nabla^2 d_t^k\|_{L^2}^2ds \le C\exp\left[C\int_0^t \Phi_K^N ds\right]
\end{equation}
for all $k, 1\le k \le K$.
\end{lemm}
\begin{proof}
Taking $\nabla$ operator to $\eqref{3.19}_4$, multiplying by $\nabla \Delta d_t^k$ and integrating over $\Omega$,
we have
\begin{equation}\label{3.38}
\begin{aligned}
&\quad \frac{1}{2}\frac{d}{dt}\int |\nabla\Delta d^k|^2 dx+\int |\Delta d_t^k|^2 dx\\
&=\frac{d}{dt}\int \nabla (u^{k-1}\cdot \nabla d^{k-1}-|\nabla d^{k-1}|^2 d^{k-1})\cdot \nabla \Delta d^k dx\\
& \quad \quad -\int \frac{\partial }{\partial t}\left[\nabla(u^{k-1}\cdot \nabla d^{k-1})-\nabla(|\nabla d^{k-1}|^2d^{k-1})\right]\cdot \nabla \Delta d^k dx.
\end{aligned}
\end{equation}
Now we need to estimate the second term of right hand side of \eqref{3.38} as follows:
\begin{equation}\label{3.39}
\begin{aligned}
&\quad -\int \frac{\partial }{\partial t}\left[\nabla(u^{k-1}\cdot \nabla d^{k-1})-\nabla(|\nabla d^{k-1}|^2d^{k-1})\right]\cdot \nabla \Delta d^k dx\\
&\le C\int \left(|\nabla u_t^{k-1}| |\nabla d^{k-1}|+|\nabla u^{k-1}| |\nabla d_t^{k-1}|
      +|u_t^{k-1}| |\nabla^2 d^{k-1}|+|u^{k-1}| |\nabla^2 d_t^{k-1}| \right.\\
&\left.\quad \quad \quad \quad +|\nabla d_t^{k-1}| |\nabla^2 d^{k-1}|+|\nabla d^{k-1}| |\nabla^2 d_t^{k-1}|
        +|\nabla d^{k-1}| |\nabla^2 d^{k-1}| |d_t^{k-1}|\right.\\
&\left.\quad \quad \quad \quad  +|\nabla d^{k-1}|^2 |\nabla d_t^{k-1}|\right)|\nabla \Delta d^k|dx
=\sum_{i=1}^8 I_{6i}.
\end{aligned}
\end{equation}
To estimate each term $I_{6i}(1\le i \le 8)$, applying H\"{o}lder, Gagliardo-Nirenberg and Young inequalities, we obtain
\begin{equation*}
\begin{aligned}
I_{61} &\le \|\nabla d^{k-1}\|_{L^\infty} \|\nabla u_t^{k-1}\|_{L^2} \|\nabla^3 d^k\|_{L^2}
      \le C\|\nabla d^{k-1}\|_{H^2} \|\nabla u_t^{k-1}\|_{L^2} \|\nabla^3 d^k\|_{L^2}\\
     &\le C\Phi_K^4+C\|\nabla u_t^{k-1}\|_{L^2}^2,\\
I_{62} &\le \|\nabla u^{k-1}\|_{L^3} \|\nabla d_t^{k-1}\|_{L^6}  \|\nabla^3 d^k\|_{L^2}
      \le C\|\nabla u^{k-1}\|_{H^1} \|\nabla^2 d_t^{k-1}\|_{L^2} \|\nabla^3 d^k\|_{L^2}\\
     &\le C(\delta)\|\nabla u^{k-1}\|_{H^1}^2 \|\nabla^3 d^k\|_{L^2}^2+\delta\|\nabla^2 d_t^{k-1}\|_{L^2}^2\\
     &\le C(\delta)(1+\|\sqrt{\rho^k}u_t^k\|_{L^2}^2)\Phi_K^{2N_1+2}+\delta\|\nabla^2 d_t^{k-1}\|_{L^2}^2,\\
I_{63}&\le \|u_t^{k-1}\|_{L^6} \|\nabla^2 d^{k-1}\|_{L^3} \|\nabla^3 d^k\|_{L^2}
      \le C\|\nabla u_t^{k-1}\|_{L^2} \|\nabla^2 d^{k-1}\|_{H^1} \|\nabla^3 d^k\|_{L^2}\\
     &\le C\Phi_K^4+C\|\nabla u_t^{k-1}\|_{L^2}^2,\\
I_{64}&\le \|u^{k-1}\|_{L^\infty} \|\nabla^2 d_t^{k-1}\|_{L^2} \|\nabla^3 d^k\|_{L^2}
      \le C\|u^{k-1}\|_{H^2}\|\nabla^3 d^k\|_{L^2}\|\nabla^2 d_t^{k-1}\|_{L^2}\\
     &\le C(\delta)\|u^{k-1}\|_{H^2}^2 \|\nabla^3 d^k\|_{L^2}^2+\delta \|\nabla^2 d_t^{k-1}\|_{L^2}^2\\
     &\le C(\delta)(1+\|\sqrt{\rho^k}u_t^k\|_{L^2}^2)\Phi_K^{2N_1+2}+\delta \|\nabla^2 d_t^{k-1}\|_{L^2}^2,\\
I_{65} &\le \|\nabla d_t^{k-1}\|_{L^6} \|\nabla^2 d^{k-1}\|_{L^3} \|\nabla^3 d^k\|_{L^2}
      \le C\|\nabla^2 d_t^{k-1}\|_{L^2} \|\nabla^2 d^{k-1}\|_{H^1} \|\nabla^3 d^k\|_{L^2}\\
     &\le C(\delta)\Phi_K^4+\delta \|\nabla^2 d_t^{k-1}\|_{L^2}^2,\\
I_{66}&\le \|\nabla d^{k-1}\|_{L^\infty} \|\nabla^2 d_t^{k-1}\|_{L^2} \|\nabla^3 d^k\|_{L^2}
      \le C\|\nabla d^{k-1}\|_{H^2} \|\nabla^3 d\|_{L^2} \|\nabla^2 d_t^{k-1}\|_{L^2}\\
     &\le C(\delta)\Phi_K^4+\delta \|\nabla^2 d_t^{k-1}\|_{L^2}^2,\\
I_{67} &\le \|\nabla d^{k-1}\|_{L^\infty} \|\nabla^2 d^{k-1}\|_{L^3} \|d_t^{k-1}\|_{L^6}\|\nabla^3 d^k\|_{L^2}\\
     &\le C\|\nabla d^{k-1}\|_{H^2} \|\nabla^2 d^{k-1}\|_{H^1} \|d_t^{k-1}\|_{H^1}\|\nabla^3 d^k\|_{L^2}
      \le C\Phi_K^6,\\
I_{68} &\le \|\nabla d^{k-1}\|_{L^6}^2 \|\nabla d_t^{k-1}\|_{L^6} \|\nabla^3 d^k\|_{L^2}
      \le \|\nabla d^{k-1}\|_{H^1}^2 \|\nabla^2 d_t^{k-1}\|_{L^2} \|\nabla^3 d^k\|_{L^2}\\
     &\le C(\delta)\Phi_K^6+\delta \|\nabla^2 d_t^{k-1}\|_{L^2}^2,\\
\end{aligned}
\end{equation*}
where we have used \eqref{3.25}, \eqref{3.32} and \eqref{3.33}.
Substituting $I_{6i}(1\le i\le 8)$ into \eqref{3.39}, integrating \eqref{3.38} over $(0, t)$
and applying \eqref{3.30}, then we have
\begin{equation}\label{3.40}
\begin{aligned}
&\quad \frac{1}{2}\int |\nabla^3 d^k|^2dx+\int_0^t\int |\nabla^2 d_t^k|^2 dxds\\
&\le \int \left(|\nabla u^{k-1}| |\nabla d^{k-1}|+ |u^{k-1}| |\nabla^2 d^{k-1}|+ |\nabla d^{k-1}| |\nabla^2 d^{k-1}| + |\nabla d^{k-1}|^3\right)|\nabla \Delta d^k|dx\\
&\quad \quad \quad +C\exp\left[C\int_0^t \Phi_K^{N_{10}} ds\right]+\delta \int_0^t \|\nabla^2 d_t^{k-1}\|_{L^2}^2 ds\\
&= \sum_{i=1}^{4} I_{7i}+C\exp\left[C\int_0^t \Phi_K^{N_{10}} ds\right]+\delta \int_0^t \|\nabla^2 d_t^{k-1}\|_{L^2}^2 ds.\\
\end{aligned}
\end{equation}
By H\"{o}lder, Gagliardo-Nirenberg and Young inequalities, we obtain
\begin{equation*}
\begin{aligned}
I_{71} &\le  \|\nabla d^{k-1}\|_{L^\infty}  \|\nabla u^{k-1}\|_{L^2}  \|\nabla^3 d^k\|_{L^2}
        \le C\|\nabla d^{k-1}\|_{L^2}^\frac{4-n}{4}\|\nabla d^{k-1}\|_{H^2}^{\frac{n}{4}}
              \|\nabla u^{k-1}\|_{L^2}  \|\nabla^3 d^k\|_{L^2}\\
       &\le C\|\nabla d^{k-1}\|_{H^1}\|\nabla u^{k-1}\|_{L^2}  \|\nabla^3 d^k\|_{L^2}
               +C\|\nabla d^{k-1}\|_{L^2}^\frac{4-n}{4} \|\nabla^3 d^{k-1}\|_{L^2}^\frac{n}{4}
               \|\nabla u^{k-1}\|_{L^2} \|\nabla^3 d^k\|_{L^2}\\
       &\le \varepsilon\|\nabla^3 d^k\|_{L^2}^2+\eta \|\nabla^3 d^{k-1}\|_{L^2}^2
              +C(\varepsilon,\eta)\left(\|\nabla d^{k-1}\|_{H^1}^2 \|\nabla u^{k-1}\|_{L^2}^2
                 +\|\nabla d^{k-1}\|_{L^2}^2 \|\nabla u^{k-1}\|_{L^2}^\frac{8}{4-n}\right),\\
\end{aligned}
\end{equation*}
\begin{equation*}
\begin{aligned}
I_{72} &\le \|u^{k-1}\|_{L^6} \|\nabla^2 d^{k-1}\|_{L^3}\|\nabla^3 d^k\|_{L^2}\\
       &\le \|\nabla u^{k-1}\|_{L^2} \|\nabla^2 d^{k-1}\|_{L^2} \|\nabla^3 d^k\|_{L^2}
               +C\|\nabla u^{k-1}\|_{L^2} \|\nabla^2 d^{k-1}\|_{L^2}^\frac{6-n}{6}
               \|\nabla^3 d^{k-1}\|_{L^2}^\frac{n}{6}  \|\nabla^3 d^k\|_{L^2}\\
       &\le \varepsilon\|\nabla^3 d^k\|_{L^2}^2+\eta \|\nabla^3 d^{k-1}\|_{L^2}^2
               +C(\varepsilon,\eta)\left(\|\nabla^2 d^{k-1}\|_{L^2}^2 \|\nabla u^{k-1}\|_{L^2}^2
                 +\|\nabla^2 d^{k-1}\|_{L^2}^2 \|\nabla u^{k-1}\|_{L^2}^{\frac{12}{6-n}}\right),\\
I_{73} &\le \|\nabla d^{k-1}\|_{L^6} \|\nabla^2 d^{k-1}\|_{L^3} \|\nabla^3 d^k\|_{L^2}
        \le  C\|\nabla d^{k-1}\|_{H^1} \|\nabla^2 d^{k-1}\|_{L^2}^\frac{6-n}{6}\|\nabla^2 d^{k-1}\|_{H^1}^\frac{n}{6}
             \|\nabla^3 d^k\|_{L^2}\\
       &\le C \|\nabla d^{k-1}\|_{H^1}^2 \|\nabla^3 d^k\|_{L^2}
              +C\|\nabla d^{k-1}\|_{H^1}^\frac{12-n}{6} \|\nabla^3 d^{k-1}\|_{L^2}^\frac{n}{6} \|\nabla^3 d^k\|_{L^2}\\
       &\le \varepsilon\|\nabla^3 d^k\|_{L^2}^2+\eta \|\nabla^3 d^{k-1}\|_{L^2}^2
              +C(\varepsilon,\eta)\left(\|\nabla d^{k-1}\|_{H^1}^4+\|\nabla d^{k-1}\|_{H^1}^{\frac{2(12-n)}{6-n}} \right),\\
I_{74} &\le \|\nabla d^{k-1}\|_{L^6}^3 \|\nabla^3 d^k\|_{L^2}
        \le   \|\nabla d^{k-1}\|_{H^1}^3 \|\nabla^3 d^k\|_{L^2}\\
       &\le \varepsilon\|\nabla^3 d^k\|_{L^2}^2+C(\varepsilon)\|\nabla d^{k-1}\|_{H^1}^6.
\end{aligned}
\end{equation*}
Substituting $I_{7i}(1\le i\le 4)$ into \eqref{3.40}, applying \eqref{3.23} and \eqref{3.35} and choosing $\varepsilon$ small enough, we obtain
\begin{equation*}
\begin{aligned}
&\quad \frac{1}{4}\int |\nabla^3 d^k|^2 dx+\int_0^t \int |\nabla^2 d_t^k|^2 dxds\\
&\le\eta\|\nabla^3 d^{k-1}\|_{L^2}^2+\delta \int_0^t \int |\nabla^2 d_t^{k-1}|^2 dxds
   +C(\delta, \eta)\exp\left[C\int_0^t \Phi_K^{N_{11}} ds\right].
\end{aligned}
\end{equation*}
Choosing $\delta$ and $\eta$ suitably small, we have
\begin{equation*}
\|\nabla^3 d^k\|_{L^2}^2+\int_0^t \|\nabla^2 d_t^k\|_{L^2}^2 ds
\le \frac{1}{2}\left(\|\nabla^3 d^{k-1}\|_{L^2}^2+\int_0^t \|\nabla^2 d_t^{k-1}\|_{L^2}^2 ds \right)
+C\exp\left[C\int_0^t \Phi_K^{N_{11}} ds\right].
\end{equation*}
By recursive relation, it arrives at
\begin{equation*}
\|\nabla^3 d^k\|_{L^2}^2+\int_0^t \|\nabla^2 d_t^k\|_{L^2}^2 ds
\le C\left(\sum_{i=1}^{k}\frac{1}{2^{i-1}}\right)\exp\left[C\int_0^t \Phi_K^{N_{11}} ds\right],
\end{equation*}
which completes the proof.
\end{proof}
Now, from Lemmas \ref{lemma3.5}-\ref{lemma3.9}, we conclude that
\begin{equation*}
\Phi_K(t)\le C \exp \left[C\exp\left(C\int_0^t \Phi_K^N ds\right)\right].
\end{equation*}
for some $N=N(n,q)>0$.
Thanks to this integral inequality, we can easily show that there exists a time $T_0\in (0, T)$(See Lemma $6$ \cite{Simon}), depending only on the parameters of $C$ such that
\begin{equation*}
\underset{0 \le t \le T_0}{\sup}\Phi_K(t) \le C.
\end{equation*}
Therefore, using the previous Lemmas \ref{lemma3.5}-\ref{lemma3.9} and other three estimates, we can derive the following
uniform bounds:
\begin{equation*}
\begin{aligned}
& \underset{0 \le t \le T_0}{\sup}\left(\|\rho^k\|_{W^{1,q}}+\|\rho^k_t\|_{L^q}+\|\sqrt{\rho^k}u_t^k\|_{L^2}
+\|u^k\|_{H^2}+\|P^k\|_{H^1}+\|d^k\|_{H^3}+\|\nabla d_t^k\|_{L^2}^2\right)\\
&\quad +\int_0^{T_0} \left(\|u^k\|_{W^{2,r}}^2+\|P^k\|_{W^{1,r}}^2+\|\nabla^4 d^k\|_{L^2}^2
+\|\nabla^2 d^k_t\|_{L^2}^2+\|d_{tt}\|_{L^2}^2\right)dt \le C,
\end{aligned}
\end{equation*}
for all $k \ge 1$.

\subsubsection{Convergence}
\quad We next show that the whole sequence $(\rho^k, u^k, d^k)$ of approximate solution converges to a solution to
the original problem \eqref{1.1}-\eqref{1.4} in a strong sense. To prove this, let us define:
\begin{equation*}
\bar{\rho}^{k+1}=\rho^{k+1}-\rho^k, \quad \bar{u}^{k+1}=u^{k+1}-u^{k} \quad \text{and} \quad \bar{d}^{k+1}=d^{k+1}-d^{k}.
\end{equation*}
\textrm{Step 1:}  It follows from the linearized momentum equation $\eqref{3.19}_2$ that
\begin{equation}\label{3.41}
\begin{aligned}
&\quad \rho^{k+1}\bar{u}^{k+1}_t+\rho^{k+1}u^k \cdot \nabla \bar{u}^{k+1}
       -{\rm div}[2 \mu(\rho^{k+1})D(\bar{u}^{k+1})]+\nabla(P^{k+1}-P^k)\\
&=-\bar{\rho}^{k+1} u_t^k-\bar{\rho}^{k+1} u^k \cdot \nabla u^k-\rho^k \bar{u}^k \cdot \nabla u^k
  +{\rm div}\left[2(\mu(\rho^{k+1})-\mu(\rho^k))D(u^k)\right]\\
&\quad -{\rm div}(\nabla \bar{d}^{k+1} \odot \nabla d^{k+1}
   +\nabla d^k \odot \nabla \bar{d}^{k+1}).
\end{aligned}
\end{equation}
Hence multiplying \eqref{3.41} by $\bar{u}^{k+1}$ and integrating (by parts) over $\Omega$, we get
\begin{equation}\label{3.42}
\begin{aligned}
&\quad \frac{1}{2}\frac{d}{dt}\int \rho^{k+1}|\bar{u}^{k+1}|^2dx+\frac{1}{C}\int |\nabla \bar{u}^{k+1}|^2 dx\\
&=\int\left\{-\bar{\rho}^{k+1} u_t^k-\bar{\rho}^{k+1} u^k \cdot \nabla u^k-\rho^k \bar{u}^k \cdot \nabla u^k
  +{\rm div}\left[2(\mu(\rho^{k+1})-\mu(\rho^k))D(u^k)\right]\right.\\
&\quad \quad \quad \left. -{\rm div}(\nabla \bar{d}^{k+1} \odot \nabla d^{k+1}
   +\nabla d^k \odot \nabla \bar{d}^{k+1})\right\} \cdot \bar{u}^{k+1}dx
=\sum_{i=1}^4 I_{8i}.
\end{aligned}
\end{equation}
To estimate $I_{8i}(1\le i\le 4)$, using H\"{o}lder, Sobolev and Young inequalities, we obtain
\begin{equation*}
\begin{aligned}
I_{81}
&\le \|\bar{\rho}^{k+1}\|_{L^2} \|u_t^k-u^k \cdot \nabla u^k\|_{L^3} \|\bar{u}^{k+1}\|_{L^6}
 \le C\|\bar{\rho}^{k+1}\|_{L^2}\|u_t^k-u^k \cdot \nabla u^k\|_{L^3} \|\nabla \bar{u}^{k+1}\|_{L^2}\\
&\le C(\varepsilon)\|\bar{\rho}^{k+1}\|_{L^2}^2\|u_t^k-u^k \cdot \nabla u^k\|_{L^3}^2
      +\varepsilon \|\nabla \bar{u}^{k+1}\|_{L^2}^2,\\
I_{82}
&\le \|\rho^k\|_{L^\infty}^{\frac{1}{2}} \|\sqrt{\rho^k}\bar{u}^k\|_{L^2} \|\nabla u^k\|_{L^3}
     \|\bar{u}^{k+1}\|_{L^6}
 \le \|\rho^k\|_{L^\infty}^{\frac{1}{2}} \|\sqrt{\rho^k}\bar{u}^k\|_{L^2} \|\nabla u^k\|_{L^3}
     \|\nabla \bar{u}^{k+1}\|_{L^2}\\
&\le  C(\varepsilon)\|\rho^k\|_{L^\infty} \|\sqrt{\rho^k}\bar{u}^k\|_{L^2}^2 \|\nabla u^k\|_{L^3}^2
      +\varepsilon \|\nabla \bar{u}^{k+1}\|_{L^2}^2,\\
I_{83}
&\le C\int |\bar{\rho}^{k+1}| |\nabla u^k| |\nabla \bar{u}^{k+1}| dx
 \le C(\varepsilon)\int |\bar{\rho}^{k+1}|^2 |\nabla u^k|^2 dx+\varepsilon\int |\nabla \bar{u}^{k+1}|^2dx\\
&\le C(\varepsilon)\|\nabla u^k\|_{W^{1,r}}^2 \|\bar{\rho}^{k+1}\|_{L^2}^2
               +\varepsilon \|\nabla \bar{u}^{k+1}\|_{L^2}^2,\\
I_{84}
&\le \int (|\nabla d^{k+1}|+|\nabla d^k|)|\nabla \bar{d}^{k+1}| |\nabla \bar{u}^{k+1}|dx\\
& \le C(\varepsilon)\int (|\nabla d^{k+1}|^2+|\nabla d^k|^2)|\nabla \bar{d}^{k+1}|^2 dx
  +\varepsilon \int |\nabla \bar{u}^{k+1}|^2 dx\\
&\le C(\varepsilon)(\|\nabla d^{k+1}\|_{H^2}^2+\|\nabla d^{k}\|_{H^2}^2)\|\nabla \bar{d}^{k+1}\|_{L^2}^2+
  +\varepsilon \int |\nabla \bar{u}^{k+1}|^2 dx.\\
\end{aligned}
\end{equation*}
Substituting $I_{8i}(1\le i\le 4)$ into \eqref{3.42} and choosing $\varepsilon$ small enough, it arrives at
\begin{equation}\label{3.43}
\begin{aligned}
&\quad \frac{1}{2}\frac{d}{dt}\int \rho^{k+1}|\bar{u}^{k+1}|^2dx+\frac{1}{2C}\int |\nabla \bar{u}^{k+1}|^2 dx\\
&\le C(\|u_t^k-u^k \cdot \nabla u^k\|_{L^3}^2+\|\nabla u^k\|_{W^{1,r}}^2)\|\bar{\rho}^{k+1}\|_{L^2}^2
    +C\|\rho^k\|_{L^\infty}\|\nabla u^k\|_{L^3}^2 \|\sqrt{\rho^k}\bar{u}^k\|_{L^2}^2\\
&\quad \quad +C(\|\nabla d^{k+1}\|_{H^2}^2+\|\nabla d^{k}\|_{H^2}^2)\|\nabla \bar{d}^{k+1}\|_{L^2}^2.
\end{aligned}
\end{equation}

\textrm{Step 2:} Since we have
$\bar{\rho}_t^{k+1}+u^k \cdot \nabla \bar{\rho}^{k+1}+\bar{u}^k \cdot \nabla \rho^k=0,$ then
multiplying by $\bar{\rho}^{k+1}$ and integrating over $\Omega$ yield
\begin{equation}\label{3.44}
\begin{aligned}
&\quad \frac{1}{2}\frac{d}{dt}\int |\bar{\rho}^{k+1}|^2 dx
  =-\int u^k \cdot \nabla\left(\frac{1}{2}|\bar{\rho}^{k+1}|^2\right)
    - \int \bar{u}^k \cdot \nabla \rho^k \cdot \bar{\rho}^{k+1}dx\\
&\le C\int |\bar{\rho}^{k+1}|^2 |\nabla u^k| dx
  +\|\nabla \rho^k\|_{L^q}\|\bar{\rho}^{k+1}\|_{L^2}\|\bar{u}^k\|_{L^{\frac{2q}{q-2}}}\\
&\le C\|\nabla u^k\|_{L^\infty}\|\bar{\rho}^{k+1}\|_{L^2}^2
  +\|\nabla \rho^k\|_{L^q}\|\bar{\rho}^{k+1}\|_{L^2}\|\nabla \bar{u}^k\|_{L^2}\\
&\le C(\delta)(\|\nabla u^k\|_{W^{1,r}}+\|\nabla \rho^k\|_{L^q}^2)\|\bar{\rho}^{k+1}\|_{L^2}^2
  +\delta \|\nabla \bar{u}^k\|_{L^2}^2,
\end{aligned}
\end{equation}
where we have used Young and Sobolev inequalities.

\textrm{Step 3:} since we have
\begin{equation}\label{3.45}
\bar{d}^{k+1}_t-\Delta \bar{d}^{k+1}
=\nabla \bar{d}^k \cdot (\nabla d^k+\nabla d^{k-1})\cdot d^k+|\nabla d^{k-1}|^2 \bar{d}^k
 -(\bar{u}^k \cdot \nabla) d^k-(u^{k-1}\cdot \nabla)\bar{d}^k.
\end{equation}
Multiplying \eqref{3.45} by $-\Delta \bar{d}^{k+1}$ and integrating (by parts) over $\Omega$, we obtain
\begin{equation}\label{3.46}
\begin{aligned}
&\quad \frac{1}{2}\frac{d}{dt}\int|\nabla \bar{d}^{k+1}|^2 dx+\int |\Delta \bar{d}^{k+1}|^2dx\\
&=\int \left[\nabla \bar{d}^{k+1}: (\nabla d^k+\nabla d^{k-1})d+|\nabla d^{k-1}|^2 \bar{d}^k
   -(\bar{u}^k \cdot)d^k-(u^{k-1 \cdot \nabla})\bar{d}^k\right]\cdot (-\Delta \bar{d}^{k+1})dx\\
&=\sum_{i=1}^4 I_{9i}.
\end{aligned}
\end{equation}
To estimate $I_{9i}(1 \le i\le 4)$, using H\"{o}lder, Young and Sobolev inequalities, we have
\begin{equation*}
\begin{aligned}
I_{91}
&\le \int |\nabla \bar{d}^k| (|\nabla d^k|+|\nabla d^{k-1}|)|d^k| |\Delta \bar{d}^{k+1}| dx\\
&\le C(\varepsilon)\int |\nabla \bar{d}^k|^2(|\nabla d^k|+|\nabla d^{k-1}|)^2 dx
     +\varepsilon\|\Delta \bar{d}^{k+1}\|_{L^2}^2\\
&\le C(\varepsilon)(\|\nabla d^k\|_{L^\infty}^2+\|\nabla d^{k-1}\|_{L^\infty}^2)\|\nabla \bar{d}^k\|_{L^2}^2
     +\varepsilon\|\Delta \bar{d}^{k+1}\|_{L^2}^2,\\
I_{92}
&\le C(\varepsilon)\int |\nabla d^{k-1}|^4 |\bar{d}^k|^2dx+\varepsilon\|\Delta \bar{d}^{k+1}\|_{L^2}^2\\
&\le C(\varepsilon)\|\nabla d^{k-1}\|_{L^6}^4 \|\bar{d}^k\|_{L^6}^2+\varepsilon\|\Delta \bar{d}^{k+1}\|_{L^2}^2\\
&\le  C(\varepsilon)\|\nabla d^{k-1}\|_{H^1}^4\|\nabla \bar{d}^k\|_{L^2}^2+\varepsilon\|\Delta \bar{d}^{k+1}\|_{L^2}^2,\\
I_{93}
&=-\int \partial_j \bar{u}_i^k \partial_i d^k \partial_j \bar{d}^{k+1}
  +\bar{u}_i^k \partial_i \partial_j d^k \partial_j \bar{d}^{k+1}dx\\
&\le C\int (|\nabla \bar{u}^k| |\nabla d^k| |\nabla \bar{d}^{k+1}| + |\bar{u}^k| |\nabla^2 d^k| |\nabla \bar{d}^{k+1}|)dx\\
&\le C\|\nabla d^k\|_{L^\infty} \|\nabla \bar{u}^k\|_{L^2} \|\nabla \bar{d}^{k+1}\|_{L^2}
     + C\| \bar{u}^k\|_{L^6}  \|\nabla^2 d^k\|_{L^3} \|\nabla \bar{d}^{k+1}\|_{L^2}\\
&\le C(\delta)\|\nabla d^k\|_{H^2}^2 \|\nabla \bar{d}^{k+1}\|_{L^2}^2
     +C(\delta)\|\nabla^2 d^k\|_{L^6}^2 \|\nabla \bar{d}^{k+1}\|_{L^2}^2
     +\delta \|\nabla \bar{u}^k\|_{L^2}^2\\
&\le  C(\delta)\|\nabla d^k\|_{H^2}^2 \|\nabla \bar{d}^{k+1}\|_{L^2}^2 +\delta \|\nabla \bar{u}^k\|_{L^2}^2,\\
I_{94}
&\le C(\varepsilon)\int |u^{k-1}|^2 |\nabla \bar{d}^k|^2 dx+\varepsilon\|\Delta \bar{d}^{k+1}\|_{L^2}^2\\
&\le C(\varepsilon)\|u^{k-1}\|_{H^2}^2 \|\nabla \bar{d}^k\|_{L^2}^2 dx+\varepsilon\|\Delta \bar{d}^{k+1}\|_{L^2}^2.\\
\end{aligned}
\end{equation*}
Substituting $I_{9i}(1 \le i\le 4)$ into \eqref{3.46} and choosing $\varepsilon$ small enough, we get
\begin{equation}\label{3.47}
\begin{aligned}
&\quad \frac{1}{2}\frac{d}{dt}\int|\nabla \bar{d}^{k+1}|^2 dx+\int |\Delta \bar{d}^{k+1}|^2dx\\
&\le C(\|\nabla d^k\|_{H^2}^2+\|\nabla d^{k-1}\|_{H^2}^2 +\|\nabla d^{k-1}\|_{H^1}^2+\|u^{k-1}\|_{H^2}^2)
       \|\nabla \bar{d}^k\|_{L^2}^2\\
&\quad \quad +C(\delta)\|\nabla d^k\|_{H^2}^2 \|\nabla \bar{d}^{k+1}\|_{L^2}^2
       +\delta \|\nabla \bar{u}^k\|_{L^2}^2.
\end{aligned}
\end{equation}
Denoting
\begin{equation*}
\begin{aligned}
& \varphi^k(t)=\|\bar{\rho}^k(t)\|_{L^2}^2+\|\sqrt{\rho^k}\bar{u}^k\|_{L^2}^2+\|\nabla \bar{d}^k\|_{L^2}^2,\\
& \psi^k(t)=\|\nabla \bar{u}^k(t)\|_{L^2}^2+\|\Delta \bar{d}^k(t)\|_{L^2}^2,\\
& F^k(t)=\|u_t^k-u^k \cdot \nabla u^k\|_{L^3}^2+\|\nabla u^k\|_{W^{1,r}}^2+\|\nabla \rho^k\|_{L^q}^2
        +\|\nabla d^k\|_{H^2}^2+\|\nabla d^{k+1}\|_{H^2}^2,\\
\end{aligned}
\end{equation*}
and
\begin{equation*}
\begin{aligned}
& G^k(t)=\|\rho^k\|_{L^\infty}^2\|\nabla u^k\|_{L^3}^2+\|\nabla d^k\|_{H^2}^2+\|\nabla^2 d^{k-1}\|_{L^2}^2+
        \|\nabla d^{k-1}\|_{H^1}^4+\|\nabla d^{k+1}\|_{H^2}^2.
\end{aligned}
\end{equation*}
Then $\eqref{3.43}+\eqref{3.44}+\eqref{3.47}$ yields
\begin{equation*}
\varphi^{k+1}(t)+\int_0^t \psi^{k+1}(s)ds
\le \int_0^t \left[C\varphi^k(s)+\delta \psi^k(s)\right]ds+\int_0^t CF^k(s)\varphi^{k+1}(s)ds,
\end{equation*}
which implies, by virtue of Gr\"{o}nwall inequality, that
\begin{equation*}
\varphi^{k+1}(t)+\int_0^t \psi^{k+1}(s)ds
\le C\exp(C_\delta t)\int_0^t \left[C\varphi^k(s)+\delta \psi^k(s)\right]ds.
\end{equation*}
Choosing $\delta>0$ and $T_{1}$ so small that $8C(T_1+\delta)<1$ and $\exp(C_\delta T_1)<2$, then
\begin{equation*}
\varphi^{k+1}(t)+\int_0^t \psi^{k+1}(s)ds
\le \frac{1}{4}\left[\varphi^{k}(t)+\int_0^t \psi^{k}(s)ds\right],
\end{equation*}
for all $1\le k \le K.$ By iteration, we have
\begin{equation*}
\varphi^{k+1}(t)+\int_0^t \psi^{k+1}(s)ds
\le \frac{1}{4^k}\left[\varphi^{1}(t)+\int_0^t \psi^{1}(s)ds\right], \quad k \ge 0.
\end{equation*}
Together with the Poincar\'{e} inequality and elliptic estimate, we get
\begin{equation*}
\begin{aligned}
&\|\bar{\rho}^{k+1}\|_{L^\infty(0,T;L^2)}+\|\sqrt{\rho^{k+1}}\bar{u}^{k+1}\|_{L^\infty(0,T;L^2)}
+\|\bar{d}^{k+1}\|_{L^\infty(0,T;H^1)}\\
&\quad +\|\bar{u}^{k+1}\|_{L^2(0,T;H^1)}
+\|\bar{d}^{k+1}\|_{L^2(0,T;H^2)}\le \frac{C}{2^k}.
\end{aligned}
\end{equation*}
Therefore, we get
\begin{equation*}
\begin{aligned}
& \sum_{k=1}^\infty \|\bar{\rho}^{k+1}\|_{L^\infty(0,T;L^2)}<\infty,\\
& \sum_{k=1}^\infty \|\bar{u}^{k+1}\|_{L^2(0,T;H^1)}<\infty,\\
& \sum_{k=1}^\infty (\|\bar{d}^{k+1}\|_{L^\infty(0,T;H^1)}+\|\bar{d}^{k+1}\|_{L^2(0,T;H^2)})<\infty,\\
\end{aligned}
\end{equation*}
which means
\begin{equation*}
\begin{aligned}
& \rho^k \rightarrow \rho^1+\sum_{i=2}^\infty \bar{\rho}^i \quad {\text in} \quad L^\infty(0,T;L^2),\\
& u^k \rightarrow u^1+\sum_{i=2}^\infty \bar{u}^i         \quad {\text in} \quad L^2(0,T;H^1),\\
& d^k \rightarrow d^1+\sum_{i=2}^\infty \bar{d}^i         \quad {\text in} \quad L^\infty(0,T;H^1)\cap L^2(0,T;H^2),\\
\end{aligned}
\end{equation*}
as $k \rightarrow \infty.$

\subsubsection{Conclusion}
\quad Now it is a simple matter to check that $(\rho, u, d)$ is a weak solution to the original problem \eqref{1.1}-\eqref{1.4}
with positive initial density. Then, by virtue of the lower semi-continuity of norms, we deduce from the uniform bound that
$(\rho, u, P, d)$ satisfies the following regularity estimate:
\begin{equation*}
\begin{aligned}
& \underset{0 \le t \le T_0}{\sup}\left(\|\rho\|_{W^{1,q}}+\|\rho_t\|_{L^q}+\|\sqrt{\rho}u_t\|_{L^2}
+\|u\|_{H^2}+\|P\|_{H^1}+\|d\|_{H^3}+\|\nabla d_t\|_{L^2}^2\right)\\
&\quad +\int_0^{T_0} \left(\|u\|_{W^{2,r}}^2+\|P\|_{W^{1,r}}^2+\|\nabla^4 d\|_{L^2}^2
+\|\nabla^2 d_t\|_{L^2}^2+\|d_{tt}\|_{L^2}^2\right)dt \le C.
\end{aligned}
\end{equation*}

\subsection{Proof of Theorem \ref{Theorem1.1}}
\quad Let $(\rho_0, u_0, d_0)$ satisfies the assumptions in Theorem \ref{Theorem1.1}. For each $\delta>0$, let
$\rho_0^\delta=\rho_0+\delta$, satisfying
\begin{center}
$\rho_0^\delta \rightarrow \rho_0$ in $W^{1,.q}$ as $\delta \rightarrow 0^+$,
\end{center}
and $d_0^\delta=d_0$. Suppose $(u^\delta, P^\delta)\in H_0^1 \times L^2$ is a solution to the problem
\begin{equation*}
-{\rm div}(2\mu(\rho_0^\delta)D(u_0^\delta))+\nabla P_0^\delta+{\rm div}(\nabla d_0^\delta \odot \nabla d_0^\delta)=\sqrt{\rho_0^\delta}g \ {\rm and } \ {\rm div}u_0^\delta=0
\ {\rm in} \ \Omega.
\end{equation*}
Then by the regularity estimate, it is easy to get
\begin{center}
$u_0^\delta \in H_0^1 \cap H^2$ \quad and \quad $u_0^\delta \rightarrow u_0$ in $H^2$ as $\delta \rightarrow 0^+.$
\end{center}
Then, by Proposition \ref{proposition3.4}, there exist a time $T_0 \in (0, T)$ and a unique strong solution
$(\rho^\delta, u^\delta, P^\delta, d^\delta)$ in $[0, T_0]\times \Omega$ to the problem with the initial data
replaced by  $(\rho^\delta_0, u^\delta_0, d^\delta_0)$. Note that $(\rho^\delta, u^\delta, P^\delta, d^\delta)$
will satisfies the regularity estimate where $C$ independent of the parameter $\delta$. Hence, let
$\delta \rightarrow 0^+$, it is a simple matter to see $(\rho, u, P, d)$ is a strong solution and have the following
regularity estimate
\begin{equation*}
\begin{aligned}
& \underset{0 \le t \le T_0}{\sup}\left(\|\rho\|_{W^{1,q}}+\|\rho_t\|_{L^q}+\|\sqrt{\rho}u_t\|_{L^2}
+\|u\|_{H^2}+\|P\|_{H^1}+\|d\|_{H^3}+\|\nabla d_t\|_{L^2}^2\right)\\
&\quad +\int_0^{T_0} \left(\|u\|_{W^{2,r}}^2+\|P\|_{W^{1,r}}^2+\|\nabla^4 d\|_{L^2}^2
+\|\nabla^2 d_t\|_{L^2}^2+\|d_{tt}\|_{L^2}^2\right)dt \le C.
\end{aligned}
\end{equation*}
Therefore, we complete the proof of Theorem \ref{Theorem1.1}.

\section{Proof of Theorem \ref{blowupcriterion1}}
\quad In this section, we will give the proof of Theorem \ref{Theorem1.1} by contradiction. More precisely, let $0 < T^* <\infty$ be the maximum time for the existence of strong solution $(\rho, u, P, d)$ to \eqref{1.1}-\eqref{1.4}. Suppose that \eqref{1.7}
were false, that is
\begin{equation}\label{4.1}
M_0 \triangleq \underset{T \rightarrow T^*}{\lim}\left(\|\nabla \rho\|_{L^\infty(0,T;L^q)}
    +\|u\|_{L^{s_1}(0,T;L_w^{r_1})}+\|\nabla d\|_{L^{s_2}(0,T;L_w^{r_2})}\right)< \infty.
\end{equation}
Under the condition \eqref{4.1}, one will extend existence time of the strong solution to \eqref{1.1}-\eqref{1.4} beyond $T^*$,
which contradicts with the definition of maximum existence time.

\begin{lemm}\label{lemma4.1}
For any $0\le T < T^*$, it holds that
\begin{equation}\label{4.2}
\underset{0 \le t \le T}{\sup}
\left(\|\rho\|_{L^m}+\|\sqrt{\rho}u\|_{L^2}^2+\|\nabla d\|_{L^2}^2\right)+\int_0^T \int \left(|\nabla d|^2+\left||\nabla d|^2d+\Delta d\right|^2\right)dxds \le C,
\end{equation}
where $1\le m \le \infty$ and in what follows, $C$ denotes generic constants depending only on $\Omega, M_0,  T^*$ and the initial data.
\end{lemm}

\begin{proof}
Multiplying $\eqref{1.1}_1$ by $m\rho^{m-1}(1\le m \le \infty)$ and integrating the resulting equation over $\Omega$,
then it is easy to deduce that
\begin{equation*}
\|\rho(t)\|_{L^m}=\|\rho_0\|_{L^m}  \quad  (1\le m \le \infty).
\end{equation*}
Multiplying $\eqref{1.1}_2$ by $u$ and integrating (by parts) over $\Omega$, it is easy to deduce
\begin{equation}\label{4.3}
\frac{1}{2} \frac{d}{dt}\int \rho |u|^2dx+\frac{1}{C}\int |\nabla u|^2 dx\le -\int (u \cdot \nabla)d \cdot \Delta d \ dx.
\end{equation}
Multiplying $\eqref{1.1}_4$ by $\Delta d+ |\nabla d|^2 d$ and integrating (by parts) over $\Omega$, we obtain
\begin{equation}\label{4.4}
\frac{1}{2} \frac{d}{dt}\int |\nabla d|^2 dx+\int \left||\nabla d|^2 d+\Delta d\right|^2dx
=\int (u\cdot \nabla)d \cdot \Delta d dx+\int(d_t+u\cdot \nabla d)|\nabla d|^2 d \ dx.
\end{equation}
By virtue of $|d|=1,$ we have the fact
\begin{equation}\label{4.5}
(d_t+u\cdot \nabla d)|\nabla d|^2 d=0.
\end{equation}
Substituting \eqref{4.5} into \eqref{4.4}, it arrives at
\begin{equation*}
\frac{1}{2} \frac{d}{dt}\int |\nabla d|^2 dx+\int \left||\nabla d|^2 d+\Delta d\right|^2dx
=\int (u\cdot \nabla)d \cdot \Delta d dx,
\end{equation*}
which, together with \eqref{4.3}, gives
\begin{equation}\label{4.6}
\frac{1}{2} \frac{d}{dt}\int (\rho |u|^2 +|\nabla d|^2 )dx+\frac{1}{C}\int |\nabla u|^2 dx
+\int \left||\nabla d|^2 d+\Delta d\right|^2 dx \le 0.
\end{equation}
Integrating \eqref{4.6} over $(0, t)$ yields
\begin{equation*}
\frac{1}{2}\int (\rho |u|^2+|\nabla d|^2)dx+\int_0^t \int
 \left(\frac{1}{C}|\nabla u|^2+\left||\nabla d|^2d+\Delta d\right|^2\right)dxd\tau \le
 \frac{1}{2}\int \left(\rho_0 |u_0|^2+|\nabla d_0|^2\right)dx,
\end{equation*}
which completes the proof.
\end{proof}
Now, we give the estimate for $\|\nabla u\|_{L^2}$ and $\|\nabla^2 d\|_{L^2}$.
\begin{lemm}\label{lemma4.2}
Under the condition \eqref{4.1}, it holds that for $0\le T <T^*$,
\begin{equation}\label{4.7}
\begin{aligned}
&\underset{0 \le t \le T}{\sup}\left(\|\nabla u\|_{L^2}^2+\|\nabla d\|_{L^4}^4+\|\nabla^2 d\|_{L^2}^2\right)\\
&\quad +\int_0^T \int \left(|\sqrt{\rho}\dot{u}|^2+|\nabla d|^2 |\nabla^2 d|^2 +|\nabla^3 d|^2+|\nabla d_t|^2\right)dxd\tau \le C.
\end{aligned}
\end{equation}
\end{lemm}
\begin{proof}

\textrm{Step 1:}
Multiplying $\eqref{1.1}_1$ by $u_t$ and integrating (by parts) over $\Omega$, we have
\begin{equation}\label{4.8}
\begin{aligned}
&\quad \frac{d}{dt}\int \mu(\rho) |D(u)|^2 dx-\frac{d}{dt}\int \nabla d \odot \nabla d : \nabla u dx+\int \rho |\dot{u}|^2 dx\\
&=-\int \rho \dot{u}\cdot (u\cdot \nabla u)+\mu' (u \cdot \nabla \rho)|D(u)|^2+ \nabla d_t \odot \nabla d: \nabla u
+ \nabla d \odot \nabla d_t :\nabla u dx.
\end{aligned}
\end{equation}
For any $0 \le s <t \le T,$ integrating over $(s, t)$ and applying the Cauchy inequality yield
\begin{equation*}
\begin{aligned}
&\quad \frac{1}{2C}\int |\nabla u|^2 dx- \int \nabla d \odot \nabla d : \nabla u dx+\int_s^t \int \rho |\dot{u}|^2dxd\tau\\
&\le C\int |\nabla u|^2 (s)dx+\int |\nabla d|^2 |\nabla u|(s)dx
+C(\varepsilon, \delta)\int_s^t \int \left(\rho |u|^2 |\nabla u|^2+|u| |\nabla \rho| |\nabla u|^2\right.\\
&\quad \quad \left.+|\nabla d|^2 |\nabla u|^2\right)dxd\tau
+\varepsilon\int_s^t \int \rho |\dot{u}|^2 dxd\tau+\delta \int_s^t \int |\nabla d_t|^2 dxd\tau.
\end{aligned}
\end{equation*}
Choosing $\varepsilon=\frac{1}{4}$, we obtain
\begin{equation}\label{4.9}
\begin{aligned}
&\quad \frac{1}{2C}\int |\nabla u|^2dx- \int \nabla d \odot \nabla d : \nabla u dx+\frac{3}{4}\int_s^t \int \rho |\dot{u}|^2dxd\tau\\
&\le C\int (|\nabla u|^2+|\nabla d|^4)(s)dx
+C(\delta)\int_s^t \int \left(\rho|u|^2 |\nabla u|^2+|u| |\nabla \rho| |\nabla u|^2\right.\\
&\quad \quad
\left.+|\nabla d|^2 |\nabla u|^2\right)dxd\tau+\delta \int_s^t \int |\nabla d_t|^2 dxd\tau.
\end{aligned}
\end{equation}
Estimate for the term $\int_s^t \int |u| |\nabla \rho| |\nabla u|^2 dxd\tau.$ Indeed, by Cauchy inequality, H\"{o}leder inequality, and Sobolev inequality, we obtain
\begin{equation*}
\begin{aligned}
&\quad \int_s^t \int |u| |\nabla \rho| |\nabla u|^2 dxd\tau
\le C(\varepsilon)\int_s^t \int|u|^2 |\nabla u|^2dxd\tau+\varepsilon \int_s^t \int|\nabla \rho|^2 |\nabla u|^2dxd\tau\\
&\le C(\varepsilon)\int_s^t \int|u|^2 |\nabla u|^2dxd\tau
+\varepsilon \int_s^t \left\||\nabla \rho|^2\right\|_{L^\frac{q}{2}} \left\||\nabla u|^2\right\|_{L^\frac{q}{q-2}}d\tau\\
&\le C(\varepsilon)\int_s^t \int|u|^2 |\nabla u|^2dxd\tau
+\varepsilon \int_s^t \|\nabla u\|_{H^1}^2 d\tau.
\end{aligned}
\end{equation*}
In order to control the term $\int_s^t \|\nabla u\|_{H^1}^2 d\tau$, by virtue of the Lemma \ref{Lemma2.1}, we have
\begin{equation}\label{4.10}
\begin{aligned}
\|u\|_{H^2}+\|P\|_{H^1}
& \le C \|F\|_{L^2}(1+\|\nabla \rho\|_{L^q})^{\frac{q}{q-n}}\\
& \le C \|-\rho u_t -\rho u \cdot \nabla u- {\rm div}(\nabla d \odot \nabla d)\|_{L^2}\\
& \le C(\|\sqrt{\rho}\dot{u}\|_{L^2}+\left\| |\nabla d| |\nabla^2 d|\right\|_{L^2}).
\end{aligned}
\end{equation}
Hence we get the estimate
\begin{equation}\label{4.11}
\begin{aligned}
&\quad \int_s^t \int |u| |\nabla \rho| |\nabla u|^2 dxd\tau\\
&\le C(\varepsilon)\int_s^t \int|u|^2 |\nabla u|^2dxd\tau
+\varepsilon \int_s^t (\|\sqrt{\rho}\dot{u}\|_{L^2}^2+\left\| |\nabla d| |\nabla^2 d|\right\|_{L^2}^2)d\tau.
\end{aligned}
\end{equation}
Substituting \eqref{4.11} into \eqref{4.9} and choosing $\varepsilon=\frac{1}{4}$, it arrives at
\begin{equation}\label{4.12}
\begin{aligned}
&\quad \frac{1}{2C}\int |\nabla u|^2dx- \int \nabla d \odot \nabla d : \nabla u dx+\frac{1}{2}\int_s^t \int \rho |\dot{u}|^2dxd\tau\\
&\le C\int (|\nabla u|^2+|\nabla d|^4)(s)dx
+C(\delta)\int_s^t \int \left(|u|^2 |\nabla u|^2+|\nabla d|^2 |\nabla u|^2\right. \\
&\quad \quad \left.+ |\nabla d|^2 |\nabla^2 d|^2\right)dxd\tau+\delta \int_s^t \int |\nabla d_t|^2 dxd\tau.
\end{aligned}
\end{equation}

\textrm{Step 2:} Taking $\nabla $ operator to $\eqref{1.1}_4$, then we have
\begin{equation}\label{4.13}
\nabla d_t -\nabla \Delta d=-\nabla (u\cdot \nabla d)+\nabla (|\nabla d|^2 d).
\end{equation}
Multiplying \eqref{4.13} by $4|\nabla d|^2 \nabla d$ and integrating (by parts) over $\Omega$, we obtain
\begin{equation}\label{4.14}
\begin{aligned}
&\quad \frac{d}{dt}\int |\nabla d|^4dx+4\int |\nabla d|^2 |\nabla^2 d|^2dx+2\int |\nabla(|\nabla d|^2)|^2 dx\\
&=2\int_{\partial \Omega} |\nabla d|^2<\nabla (|\nabla d|^2), \nu>d\sigma
+4\int \nabla(|\nabla d|^2 d):|\nabla d|^2 \nabla d \ dx\\
&\quad \quad -4\int \nabla (u \cdot \nabla d):|\nabla d|^2 \nabla d \ dx=\sum_{i=1}^3 {I\!I}_{1i},
\end{aligned}
\end{equation}
where $\nu $ is the unite outward normal vector to $\partial \Omega$.
To estimate ${I\!I}_{11}=2\int_{\partial \Omega} |\nabla d|^2<\nabla (|\nabla d|^2), \nu>d\sigma$. Indeed, applying the
Sobolev embedding inequality $W^{1,1}(\Omega)\hookrightarrow L^1(\partial \Omega)$, it is easy to get
\begin{equation}\label{4.15}
\begin{aligned}
{I\!I}_{11}
& \le 4 \int_{\partial \Omega} |\nabla d|^3 |\nabla^2 d| d\sigma
  \le C \left\| |\nabla d|^3 |\nabla^2 d| \right\|_{W^{1,1}(\Omega)}\\
& \le C\int ( |\nabla d|^3 |\nabla^2 d|+ |\nabla d|^2 |\nabla^2 d|^2+|\nabla d|^3 |\nabla^3 d|)dx\\
& \le C(\eta)\int (|\nabla d|^2 |\nabla^2 d|^2+|\nabla d|^4+|\nabla d|^6)dx+\eta \int |\nabla^3 d|^2 dx.\\
\end{aligned}
\end{equation}
To estimate ${I\!I}_{12}=4\int \nabla(|\nabla d|^2 d):|\nabla d|^2 \nabla d \ dx.$ Indeed, since $|d|=1$, we have
$d \cdot \nabla d=0.$ Then, we get
\begin{equation*}
\nabla(|\nabla d|^2 d):|\nabla d|^2 \nabla d
=(\nabla(|\nabla d|^2)d+|\nabla d|^2 \nabla d):|\nabla d|^2 \nabla d=|\nabla d|^6.
\end{equation*}
Hence, it arrives at
\begin{equation}\label{4.16}
{I\!I}_{12}=4\int \nabla(|\nabla d|^2 d):|\nabla d|^2 \nabla d \ dx=4\int |\nabla d|^6 dx.
\end{equation}
By the Cauchy inequality, we have
\begin{equation}\label{4.17}
\begin{aligned}
{I\!I}_{13}&=-4\int \nabla (u \cdot \nabla d):|\nabla d|^2 \nabla d \ dx\\
     &\le \int (|\nabla u| |\nabla d|^4+ |u| |\nabla d|^3 |\nabla^2 d|)dx\\
     &\le C\int (|\nabla d|^2 |\nabla u|^2 +|\nabla d|^6 +|u|^2 |\nabla^2 d|^2) dx.
\end{aligned}
\end{equation}
Substituting \eqref{4.15}-\eqref{4.17} into \eqref{4.14}, choosing $\varepsilon$ small enough, and integrating
over $(s, t)$, it arrives at
\begin{equation}\label{4.18}
\begin{aligned}
&\quad \int |\nabla d|^4 dx+4\int_s^t \int |\nabla d|^2 |\nabla^2 d|^2dx d\tau+2\int_s^t \int |\nabla (|\nabla d|^2)|^2dxd\tau\\
&\le \int |\nabla d|^4(s)dx+C(\eta)\int_s^t \int\left(|\nabla d|^4+|\nabla d|^6+|\nabla d|^2 |\nabla^2 d|^2+|\nabla d|^2 |\nabla u|^2\right.\\
&\quad \quad \left.+ |u|^2 |\nabla^2 d|^2\right)dxd\tau +\eta \int_s^t \int |\nabla^3 d|^2 dxd\tau.
\end{aligned}
\end{equation}

\textrm{Step 3:}Multiplying \eqref{4.13} by $\nabla \Delta d$, integrating (by parts) over $\Omega$ and applying
Young inequality, we obtain
\begin{equation*}
\begin{aligned}
&\quad \frac{1}{2}\frac{d}{dt}\int |\Delta d|^2dx+\int |\nabla \Delta d|^2dx\\
&\le \int (|\nabla u| |\nabla d|+ |u| |\nabla^2 d| + |\nabla d|^3 + |\nabla d| |\nabla^2 d|)|\nabla \Delta d|dx\\
&\le C(\varepsilon)\int (|\nabla u|^2 |\nabla d|^2+|u|^2 |\nabla^2 d|^2+|\nabla d|^6 +|\nabla d|^2 |\nabla^2 d|^2)dx
+\varepsilon\int |\nabla \Delta d|^2 dx.
\end{aligned}
\end{equation*}
Choosing $\varepsilon=\frac{1}{2}$ and integrating over $(s, t)$, we get
\begin{equation}\label{4.19}
\begin{aligned}
&\quad \int |\Delta d|^2dx+\int_s^t\int |\nabla \Delta d|^2dxd\tau\\
&\le 2\int |\Delta d|^2(s)dx +C\int_s^t\int (|\nabla d|^2 |\nabla^2 d|^2+|\nabla d|^2|\nabla u|^2+|u|^2 |\nabla^2 d|^2+|\nabla d|^6)dxd\tau.
\end{aligned}
\end{equation}

\textrm{Step 4:}Multiplying \eqref{4.13} by $\nabla d_t$, integrating (by parts) over $\Omega$ and applying
Young inequality, we obtain
\begin{equation*}
\begin{aligned}
&\quad \frac{1}{2}\frac{d}{dt}\int |\Delta d|^2dx+\int | \nabla d_t|^2dx
  =\int [\nabla(|\nabla d|^2 d)-\nabla(u\cdot \nabla d)]\cdot \nabla d_t dx\\
&\le C(\varepsilon)\int (|\nabla u|^2 |\nabla d|^2+|u|^2 |\nabla^2 d|^2+|\nabla d|^6 +|\nabla d|^2 |\nabla^2 d|^2)dx
+\varepsilon\int |\nabla d_t|^2 dx.
\end{aligned}
\end{equation*}
Choosing $\varepsilon=\frac{1}{2}$ and integrating over $(s, t)$, we get
\begin{equation}\label{4.20}
\begin{aligned}
&\quad \int |\Delta d|^2dx+\int_s^t\int |\nabla d_t|^2dxd\tau\\
&\le 2\int |\Delta d|^2(s)dx +C\int_s^t\int (|\nabla d|^2 |\nabla^2 d|^2+|\nabla u|^2 |\nabla d|^2+|u|^2 |\nabla^2 d|^2+|\nabla d|^6)dxd\tau.
\end{aligned}
\end{equation}
Then, choosing $\delta$ small enough and some constant $C_{**}$ suitably large such that
\begin{equation*}
\frac{1}{2C}|\nabla u|^2-\nabla d \odot \nabla d :\nabla u+C_{**} |\nabla d|^4 \ge \frac{1}{4C}(|\nabla u|^2+|\nabla d|^4),
\end{equation*}
then $\eqref{4.12}+\eqref{4.18}\times C_{**}+\eqref{4.19}+\eqref{4.20} $ and choosing $\eta$ and $\delta$ small enough yield
\begin{equation}\label{4.21}
\begin{aligned}
&\quad \int(|\nabla u|^2+|\nabla^2 d|^2+|\nabla d|^4)(t)dx
    +\int_s^t\int(\rho |\dot{u}|^2+|\nabla d|^2 |\nabla^2 d|^2+|\nabla(|\nabla d|^2)|^2\left.\right.\\
&\quad \quad +|\nabla^3 d|^2+|\nabla d_t|^2)dxd\tau\\
&\le C+C\int(|\nabla u|^2+|\nabla^2 d|^2+|\nabla d|^4)(s)dx
     +C\int_s^t \int(1+|\nabla^2 d|^2+|\nabla d|^4\\
&\quad \quad +|u|^2 |\nabla u|^2+|\nabla d|^2 |\nabla^2 d|^2
     +|\nabla d|^2 |\nabla u|^2+|u|^2 |\nabla^2 d|^2+|\nabla d|^6 )dxd\tau.
\end{aligned}
\end{equation}
Choosing $s=0$, then we obtain
\begin{equation}\label{4.22}
\begin{aligned}
&\quad \int(|\nabla u|^2+|\nabla^2 d|^2+|\nabla d|^4)(t)dx
    +\int_0^t\int\left(\rho |\dot{u}|^2+|\nabla d|^2 |\nabla^2 d|^2+|\nabla(|\nabla d|^2)|^2\right.\\
&\quad \quad \left.+|\nabla^3 d|^2+|\nabla d_t|^2\right)dxd\tau\\
&\le C+C\int_0^t \int\left(1+|\nabla^2 d|^2+|\nabla d|^4+|\nabla d|^2 |\nabla^2 d|^2+|u|^2 |\nabla u|^2
  +|u|^2 |\nabla^2 d|^2\right.\\
&\quad \quad \left.+|\nabla d|^2 |\nabla u|^2+|\nabla d|^6 \right)dxd\tau\\
&=C+C\int_0^t \int(1+|\nabla^2 d|^2+|\nabla d|^4)dxd\tau+\sum_{i=1}^5 {I\!I}_{2i}.
\end{aligned}
\end{equation}
Applying Lemma \ref{lemma2.3} to ${I\!I}_{2i}(1\le i\le 5)$ repeatedly, then it arrives at
\begin{equation*}
\begin{aligned}
{I\!I}_{21}&\le \varepsilon\|\nabla^2 d\|_{H^1}^2
                +C(\varepsilon)\|\nabla d\|_{L^{r_2}_w}^{\frac{2r_2}{r_2-3}}\|\nabla^2 d\|_{L^2}^2,\\
{I\!I}_{22}&\le \varepsilon\int_0^t \|\nabla u\|_{H^1}^2 d\tau
           +C(\varepsilon)\int_0^t \|u\|_{L_w^{r_1}}^\frac{2 r_1}{r_1-3}\|\nabla u\|_{L^2}^2 d\tau\\
       &\le \varepsilon\int_0^t (\|\sqrt{\rho}\dot{u}\|_{L^2}^2+\left\| |\nabla d| |\nabla^2 d|\right\|_{L^2}^2) d\tau
           +C(\varepsilon)\int_0^t \|u\|_{L_w^{r_1}}^\frac{2 r_1}{r_1-3}\|\nabla u\|_{L^2}^2 d\tau,\\
       &\le \varepsilon\int_0^t (\|\sqrt{\rho}\dot{u}\|_{L^2}^2+\|\nabla^2 d\|_{H^1}^2) d\tau
           +C(\varepsilon)\int_0^t \left(\|\nabla d\|_{L^{r_2}_w}^{\frac{2r_2}{r_2-3}}\|\nabla^2 d\|_{L^2}^2
                +\|u\|_{L_w^{r_1}}^\frac{2 r_1}{r_1-3}\|\nabla u\|_{L^2}^2 \right)d\tau,\\
{I\!I}_{23}&\le \varepsilon\int_0^t \|\nabla^2 d\|_{H^1}^2 d\tau
           +C(\varepsilon)\int_0^t \|u\|_{L_w^{r_1}}^\frac{2 r_1}{r_1-3}\|\nabla^2 d\|_{L^2}^2 d\tau,\\
{I\!I}_{24}&\le \varepsilon\int_0^t \|\nabla u\|_{H^1}^2 d\tau
           +C(\varepsilon)\int_0^t \|\nabla d\|_{L_w^{r_2}}^\frac{2 r_2}{r_2-3}\|\nabla u\|_{L^2}^2 d\tau\\
       &\le \varepsilon\int_0^t (\|\sqrt{\rho}\dot{u}\|_{L^2}^2+\|\nabla^2 d\|_{H^1}^2) d\tau
           +C(\varepsilon)\int_0^t \|\nabla d\|_{L_w^{r_2}}^\frac{2 r_2}{r_2-3}
                  \left(\|\nabla^2 d\|_{L^2}^2+\|\nabla u\|_{L^2}^2\right)d\tau,\\
{I\!I}_{25}&\le \varepsilon\int_0^t \||\nabla d|^2\|_{H^1}^2 d\tau
           +C(\varepsilon)\int_0^t \|\nabla d\|_{L_w^{r_2}}^\frac{2 r_2}{r_2-3}\||\nabla d|^2\|_{L^2}^2 d\tau\\
      &\le \varepsilon\int_0^t \left(\|\nabla d\|_{L^4}^4+\|\nabla(|\nabla d|^2)\|_{L^2}^2\right) d\tau
           +C(\varepsilon)\int_0^t \|\nabla d\|_{L_w^{r_2}}^\frac{2 r_2}{r_2-3}\|\nabla d\|_{L^4}^4 d\tau.\\
\end{aligned}
\end{equation*}
Substituting ${I\!I}_{2i}(1\le i \le 5)$ into \eqref{4.22} and choosing $\varepsilon$ small enough, then it is easy to deduce that
\begin{equation*}
\begin{aligned}
&\quad \int(|\nabla u|^2+|\nabla^2 d|^2+|\nabla d|^4)(t)dx
    +\int_0^t\int\left(\rho |\dot{u}|^2+|\nabla d|^2 |\nabla^2 d|^2+|\nabla(|\nabla d|^2)|^2\right.\\
&\quad \quad \left.+|\nabla^3 d|^2+|\nabla d_t|^2\right)dxd\tau\\
&\le C+C\int_0^t \left(1+\|u\|_{L_w^{r_1}}^\frac{2r_1}{r_1-3}+\|\nabla d\|_{L_w^{r_2}}^\frac{2r_2}{r_2-3}\right)
                     \left(1+\|\nabla u\|_{L^2}^2+\|\nabla^2 d\|_{L^2}^2+\|\nabla d\|_{L^4}^4\right)d\tau,
\end{aligned}
\end{equation*}
which, applying \eqref{4.1} and Gr\"{o}nwall inequality, completes the proof.
\end{proof}

As a corollary of Lemma \ref{lemma4.2}, it is a direct result from $\eqref{1.1}_4$. More precisely,
\begin{coro}\label{lemma4.3}
Under the condition \eqref{4.1}, it holds that for $0 \le T<T^*$,
\begin{equation}\label{4.23}
\underset{0 \le t \le T}{\sup} \|d_t\|_{L^2}+\int_0^T \| u \|_{H^2}^2 d\tau \le C.
\end{equation}
\end{coro}

Now, we give the second important estimate$-$norm of $\|\sqrt{\rho}u_t\|_{L^2}$ and $\|\nabla^3 d\|_{L^2}$.
\begin{lemm}\label{lemma4.4}
Under the condition \eqref{4.1}, it holds that for $0 \le T<T^*$,
\begin{equation}\label{4.24}
\underset{0 \le t \le T}{\sup}\left(\|\sqrt{\rho}u_t\|_{L^2}^2+\|\nabla^3 d\|_{L^2}^2\right)
+\int_0^t (\|\nabla u_t\|_{L^2}^2+\|\nabla^2 d_t\|_{L^2}^2) d\tau \le C.
\end{equation}
\end{lemm}

\begin{proof}
\textrm{Step 1:} Differentiating $\eqref{1.1}_2$ with respect to $t$, we get
\begin{equation}\label{4.25}
\begin{aligned}
&\quad \rho u_{tt}+\rho u\cdot \nabla u_t-{\rm div}(2\mu(\rho)D(u_t))+\nabla P_t\\
&=-\rho_t u_t-\rho_t u\cdot \nabla u-\rho u_t \cdot \nabla u+{\rm div}(2\mu' \rho_t D(u))
-{\rm div}(\nabla d_t \odot \nabla d+\nabla d \odot \nabla d_t).
\end{aligned}
\end{equation}
Multiplying \eqref{4.25} by $u_t$ and integrating (by parts) over $\Omega$, it arrives at
\begin{equation}\label{4.26}
\begin{aligned}
&\quad \frac{1}{2}\frac{d}{dt}\int \rho |u_t|^2dx+\frac{1}{C}\int |\nabla u_t|^2 dx\\
&\le C \int \left(\rho |u| |u_t| |\nabla u_t|+\rho |u| |\nabla u|^2 |u_t|+\rho |u|^2 |\nabla^2 u| |u_t|+\rho |u|^2 |\nabla u| |\nabla u_t|\right.\\
&\quad \quad \quad \left.+\rho |u_t|^2 |\nabla u|+|\nabla d| |\nabla d_t| |\nabla u_t|
                     +|u| |\nabla \rho| |\nabla u| |\nabla u_t|\right)dx=\sum_{i=1}^{7}{I\!I}_{3i}.
\end{aligned}
\end{equation}
Using \eqref{4.7}, H\"{o}lder, interpolation, Sobolev, and Young inequalities repeatedly, we get
\begin{equation*}
\begin{aligned}
{I\!I}_{31}&\le \|\sqrt{\rho}\|_{L^\infty} \|u\|_{L^6} \|\sqrt{\rho}u_t\|_{L^3} \|\nabla u_t\|_{L^2}
       \le C \|\nabla u\|_{L^2} \|\sqrt{\rho}u_t\|_{L^2}^\frac{1}{2} \|\sqrt{\rho}u_t\|_{L^6}^\frac{1}{2}\|\nabla u_t\|_{L^2}\\
      &\le C \|\sqrt{\rho}u_t\|_{L^2}^\frac{1}{2} \|\nabla u_t\|_{L^2}^\frac{3}{2}
       \le C(\varepsilon) \|\sqrt{\rho}u_t\|_{L^2}^2+\varepsilon \|\nabla u_t\|_{L^2}^2,\\
{I\!I}_{32}&\le \|\rho \|_{L^\infty} \|u\|_{L^6} \|\nabla u\|_{L^3}^2 \| u_t\|_{L^6}
       \le C \|\nabla u\|_{L^2}^2\|\nabla u\|_{L^6}\|\nabla u_t\|_{L^2}\\
      &\le C(\varepsilon) \|\nabla u\|_{H^1}^2+\varepsilon \|\nabla u_t\|_{L^2}^2,\\
{I\!I}_{33}&\le \|\rho\|_{L^\infty} \|u\|_{L^6} \|\nabla^2 u\|_{L^2} \| u_t\|_{L^6}
       \le C\|\nabla u\|_{L^2} \|\nabla^2 u\|_{L^2} \|\nabla u_t\|_{L^2}\\
      &\le C(\varepsilon) \|\nabla^2 u\|_{L^2}^2+\varepsilon \|\nabla u_t\|_{L^2}^2,\\
\end{aligned}
\end{equation*}
\begin{equation*}
\begin{aligned}
{I\!I}_{34}&\le \|\rho\|_{L^\infty} \|u\|_{L^6}^2 \|\nabla u\|_{L^6} \|\nabla u_t\|_{L^2}
       \le C\|\nabla u\|_{L^2}^2 \|\nabla u\|_{H^1} \|\nabla u_t\|_{L^2}\\
      &\le C(\varepsilon) \|\nabla u\|_{H^1}^2+\varepsilon \|\nabla u_t\|_{L^2}^2,\\
{I\!I}_{35}&\le  \|\sqrt{\rho}\|_{L^\infty} \|u_t\|_{L^6} \|\sqrt{\rho}u_t\|_{L^3} \|\nabla u\|_{L^2}
       \le C \|\nabla u_t\|_{L^2} \|\sqrt{\rho}u_t\|_{L^2}^\frac{1}{2}\|\sqrt{\rho}u_t\|_{L^6}^\frac{1}{2}\\
      &\le C \|\sqrt{\rho}u_t\|_{L^2}^\frac{1}{2} \|\nabla u_t\|_{L^2}^\frac{3}{2}
       \le C(\varepsilon) \|\sqrt{\rho}u_t\|_{L^2}^2+\varepsilon \|\nabla u_t\|_{L^2}^2,\\
{I\!I}_{36}&\le \|\nabla d\|_{L^6} \|\nabla d_t\|_{L^3} \|\nabla u_t\|_{L^2}
       \le C\|\nabla d\|_{H^1} \|\nabla d_t\|_{L^2}^\frac{1}{2} \|\nabla d_t\|_{L^6}^\frac{1}{2}\|\nabla u_t\|_{L^2}\\
       &\le C(\varepsilon,\eta)\|\nabla d_t\|_{L^2}^2+\eta\|\nabla d_t\|_{H^1}^2+\varepsilon\|\nabla u_t\|_{L^2}^2.
\end{aligned}
\end{equation*}
Substituting ${I\!I}_{3i}(1\le i\le 6)$ into \eqref{4.26}, then we have
\begin{equation}\label{4.27}
\begin{aligned}
&\quad \frac{1}{2} \frac{d}{dt}\int \rho |u_t|^2dx+\frac{1}{C}\int |\nabla u_t|^2 dx\\
&\le C(\varepsilon)(\|\sqrt{\rho}u_t\|_{L^2}^2+\|\nabla u\|_{H^1}^2)
     +C(\varepsilon,\eta)\|\nabla d_t\|_{L^2}^2+\eta\|\nabla d_t\|_{H^1}^2+\varepsilon\|\nabla u_t\|_{L^2}^2\\
&\quad +C\int |u| |\nabla \rho| |\nabla u| |\nabla u_t| dx.
\end{aligned}
\end{equation}
Using \eqref{4.10},\eqref{4.7} and interpolation inequality, it arrives at
\begin{equation*}
\begin{aligned}
\|u\|_{H^2}+\|P\|_{H^1}
&\le C(\|\sqrt{\rho}u_t\|_{L^2}+\| |u| |\nabla u|\|_{L^2}+\| |\nabla d| |\nabla^2 d| \|_{L^2}\\
&\le C(\|\sqrt{\rho}u_t\|_{L^2}+\|u\|_{L^6}\|\nabla u\|_{L^3}+\|\nabla d\|_{L^\infty}\|\nabla^2 d \|_{L^2})\\
&\le C(\|\sqrt{\rho}u_t\|_{L^2}+\|\nabla u\|_{L^2}^{\frac{1}{2}}\|\nabla u\|_{H^1}^{\frac{1}{2}}
       +\|\nabla d\|_{H^2})\\
&\le C(\|\sqrt{\rho}u_t\|_{L^2}+\|\nabla d\|_{H^2}+1)+\frac{1}{2}\|\nabla u\|_{H^1},\\
\end{aligned}
\end{equation*}
which implies
\begin{equation}\label{4.28}
\|u\|_{H^2}+\|P\|_{H^1} \le  C(\|\sqrt{\rho}u_t\|_{L^2}+\|\nabla d\|_{H^2}+1).
\end{equation}
Combining \eqref{4.27} with \eqref{4.28}, it arrives at
\begin{equation}\label{4.29}
\begin{aligned}
&\quad \frac{1}{2}\frac{d}{dt}\int \rho |u_t|^2dx+\frac{1}{C}\int |\nabla u_t|^2 dx\\
&\le C(\varepsilon)(\|\sqrt{\rho}u_t\|_{L^2}^2+\|\nabla^3 d\|_{L^2}^2+1)
     +C(\varepsilon,\eta)\|\nabla d_t\|_{L^2}^2+\eta\|\nabla d_t\|_{H^1}^2+\varepsilon\|\nabla u_t\|_{L^2}^2\\
&\quad +C\int |u| |\nabla \rho| |\nabla u| |\nabla u_t| dx.
\end{aligned}
\end{equation}
To estimate the term $\int |u| |\nabla \rho| |\nabla u| |\nabla u_t| dx.$ Indeed, by using \eqref{4.1}, H\"{o}lder,
Gagliardo-Nirenberg and Young inequalities, we obtain
\begin{equation}\label{4.30}
\begin{aligned}
\int |u| |\nabla \rho| |\nabla u| |\nabla u_t| dx
& \le \|u\|_{L^6} \|\nabla \rho\|_{L^q} \|\nabla u\|_{L^\frac{3q}{q-3}} \|\nabla u_t\|_{L^2}\\
& \le C\|\nabla u\|_{L^2} \|\nabla \rho\|_{L^q}
       \|\nabla u\|_{L^2}^{\frac{2q-6}{3q}} \|\nabla u\|_{L^\infty}^{\frac{q+6}{3q}} \|\nabla u_t\|_{L^2}\\
&\le C\|\nabla u\|_{W^{1,r}}^{\frac{q+6}{3q}} \|\nabla u_t\|_{L^2}.
\end{aligned}
\end{equation}
By Lemma \ref{Lemma2.1}, it is easy to deduce that
\begin{equation}\label{4.31}
\|u\|_{W^{2,r}}+\|P\|_{W^{1,r}} \le C(1+\|\nabla u_t\|_{L^2}+\|\nabla^3 d\|_{L^2}^2).
\end{equation}
By Young inequality, Combining \eqref{4.30} with \eqref{4.31} yields
\begin{equation}\label{4.32}
\begin{aligned}
\int |u| |\nabla \rho| |\nabla u| |\nabla u_t| dx
&\le C(1+\|\nabla u_t\|_{L^2}+\|\nabla^3 d\|_{L^2}^2)^{\frac{q+6}{3q}} \|\nabla u_t\|_{L^2}\\
&\le C(\varepsilon)(1+\|\nabla^3 d\|_{L^2}^4)+\varepsilon\|\nabla u_t\|_{L^2}^2.
\end{aligned}
\end{equation}
Substituting \eqref{4.32} into \eqref{4.29} and choosing $\varepsilon$ small enough, we obtain
\begin{equation*}
\begin{aligned}
&\quad \frac{1}{2}\frac{d}{dt}\int \rho |u_t|^2dx+\frac{1}{C}\int |\nabla u_t|^2 dx\\
&\le C(1+\|\sqrt{\rho}u_t\|_{L^2}^2+\|\nabla^3 d\|_{L^2}^2+\|\nabla^3 d\|_{L^2}^4)
     +C(\eta)\|\nabla d_t\|_{L^2}^2+\eta\|\nabla d_t\|_{H^1}^2.
\end{aligned}
\end{equation*}
Thanks to the compatibility condition and \eqref{4.7}, we get
\begin{equation}\label{4.33}
\begin{aligned}
&\quad \|\sqrt{\rho}u_t\|_{L^2}^2+\int_0^t \|\nabla u_t\|_{L^2}^2 d\tau\\
&\le C(\eta)+C\int_0^t(1+\|\sqrt{\rho}u_t\|_{L^2}^2+\|\nabla^3 d\|_{L^2}^2+\|\nabla^3 d\|_{L^2}^4)d\tau
     +\eta\int_0^t \|\nabla^2 d_t\|_{L^2}^2 d\tau.
\end{aligned}
\end{equation}

\textrm{Step 2:} In order to the control the term $\int_0^t \|\nabla^3 d\|_{L^2}^4 d\tau$, we should derive some estimate for
$\|\nabla^3 d\|_{L^2}^2$. More precisely, multiplying \eqref{4.13} by $\nabla \Delta d_t$ and integrating (by parts) over
$\Omega$, we obtain
\begin{equation}\label{4.34}
\begin{aligned}
&\quad \frac{1}{2}\frac{d}{dt}\int |\nabla \Delta d|^2 dx+\int |\Delta d_t|^2 dx\\
&= \frac{d}{dt}\int \nabla (u \cdot \nabla d-|\nabla d|^2 d)\cdot \nabla \Delta d \ dx
   -\int \frac{\partial}{\partial t}\left[\nabla(u\cdot \nabla d-|\nabla d|^2 d)\right] \cdot \nabla \Delta d \ dx.
\end{aligned}
\end{equation}
To give the second term of the left hand side of \eqref{4.34} first. Indeed, it is easy to deduce that
\begin{equation}\label{4.35}
\begin{aligned}
&\quad -\int \frac{\partial}{\partial t}\left[\nabla(u\cdot \nabla d-|\nabla d|^2 d)\right] \cdot \nabla \Delta d \ dx\\
&\le C\int\left(|\nabla u_t| |\nabla d|+|\nabla u| |\nabla d_t|+|u_t| |\nabla^2 d|
                 +|u| |\nabla^2 d_t|+|\nabla d_t| |\nabla^2 d|+|\nabla d| |\nabla^2 d_t|\right.\\
&\quad \quad \quad \quad  \left.+|\nabla d| |\nabla^2 d||d_t|+|\nabla d|^2 |\nabla d_t|\right)|\nabla \Delta d|dx
=\sum_{i=1}^8 {I\!I}_{4i}.
\end{aligned}
\end{equation}
By using \eqref{4.7}, \eqref{4.10}, H\"{o}lder, Sobolev and Young inequalities repeatedly, we obtain
\begin{equation*}
\begin{aligned}
{I\!I}_{41}&\le C\|\nabla d\|_{L^\infty} \|\nabla u_t\|_{L^2} \|\nabla \Delta d\|_{L^2}
       \le C\|\nabla d\|_{H^2} \|\nabla^3 d\|_{L^2}\|\nabla u_t\|_{L^2} \\
      &\le C(\delta)(\|\nabla^3 d\|_{L^2}^2+1)\|\nabla^3 d\|_{L^2}^2+\delta \|\nabla u_t\|_{L^2}^2,\\
{I\!I}_{42}&\le C\|\nabla u\|_{L^3}\|\nabla d_t\|_{L^6}\|\nabla \Delta d\|_{L^2}
        \le C\|\nabla u\|_{H^1}\|\nabla^3 d\|_{L^2}\|\nabla d_t\|_{H^1}\\
      &\le C(\varepsilon)\|\nabla u\|_{H^1}^2\|\nabla^3 d\|_{L^2}^2+\varepsilon\|\nabla d_t\|_{H^1}^2\\
      &\le C(\varepsilon)(\|\sqrt{\rho}u_t\|_{L^2}^2+\|\nabla^3 d\|_{L^2}^2+1)\|\nabla^3 d\|_{L^2}^2
           +\varepsilon\|\nabla d_t\|_{H^1}^2,\\
{I\!I}_{43}&\le \|u_t\|_{L^6} \|\nabla^2 d\|_{L^3}\|\nabla \Delta d\|_{L^2}
       \le \|\nabla u_t\|_{L^2} \|\nabla^2 d\|_{H^1}\|\nabla^3 d\|_{L^2}\\
      &\le C(\delta)(1+\|\nabla^3 d\|_{L^2}^2)\|\nabla^3 d\|_{L^2}^2+\delta \|\nabla u_t\|_{L^2}^2,\\
\end{aligned}
\end{equation*}
\begin{equation*}
\begin{aligned}
{I\!I}_{44}&\le C\|u\|_{L^\infty} \|\nabla^2 d_t\|_{L^2}\|\nabla \Delta d\|_{L^2}
       \le C\|u\|_{H^2} \|\nabla^3 d\|_{L^2} \|\nabla^2 d_t\|_{L^2}\\
      &\le (\varepsilon)(\|\sqrt{\rho}u_t\|_{L^2}^2+\|\nabla^3 d\|_{L^2}^2+1)\|\nabla^3 d\|_{L^2}^2
           +\varepsilon\|\nabla^2 d_t\|_{L^2}^2,\\
{I\!I}_{45}&\le C\|\nabla d_t\|_{L^6} \|\nabla^2 d\|_{L^3}\|\nabla \Delta d\|_{L^2}
       \le C\|\nabla d_t\|_{H^1} \|\nabla^2 d\|_{H^1}\|\nabla^3 d\|_{L^2}\\
      &\le C(\varepsilon)(1+\|\nabla^3 d\|_{L^2}^2)\|\nabla^3 d\|_{L^2}^2+\varepsilon \|\nabla d_t\|_{H^1}^2,\\
{I\!I}_{46}&\le C\|\nabla d\|_{L^6} \|\nabla^2 d\|_{L^6} \| d_t\|_{L^6}\|\nabla \Delta d\|_{L^2}
       \le C\|\nabla d\|_{H^1} \|\nabla^2 d\|_{H^1}\| d_t\|_{H^1}\|\nabla^3 d\|_{L^2}\\
      &\le C(\varepsilon)(1+\|\nabla^3 d\|_{L^2}^2)\|\nabla^3 d\|_{L^2}^2+\varepsilon( \|\nabla d_t\|_{L^2}^2+1),\\
{I\!I}_{47}&\le C\|\nabla d\|_{L^\infty} \|\nabla^2 d_t\|_{L^2}\|\nabla \Delta d\|_{L^2}
       \le C\|\nabla d\|_{H^2} \|\nabla^3 d\|_{L^2} \|\nabla^2 d_t\|_{L^2}\\
      &\le C(\varepsilon)(1+\|\nabla^3 d\|_{L^2}^2)\|\nabla^3 d\|_{L^2}^2+\varepsilon \|\nabla^2 d_t\|_{L^2}^2,\\
{I\!I}_{48}&\le C\|\nabla d\|_{L^6}^2 \|\nabla d_t\|_{L^6}\|\nabla \Delta d\|_{L^2}
       \le C\|\nabla d\|_{H^1}^2 \| \nabla d_t\|_{H^1}\|\nabla^3 d\|_{L^2}\\
      &\le C(\varepsilon)\|\nabla^3 d\|_{L^2}^2+\varepsilon \|\nabla d_t\|_{H^1}^2.\\
\end{aligned}
\end{equation*}
Substituting ${I\!I}_{4i}(1\le i \le 8)$ into \eqref{4.35}, then we obtain
\begin{equation}\label{4.36}
\begin{aligned}
&\quad -\int \frac{\partial}{\partial t}\left[\nabla(u\cdot \nabla d-|\nabla d|^2 d)\right] \cdot \nabla \Delta d \ dx\\
&\le C(\varepsilon, \delta)(1+\|\sqrt{\rho}u_t\|_{L^2}^2+\|\nabla^3 d\|_{L^2}^2)\|\nabla^3 d\|_{L^2}^2
       +\varepsilon \|\nabla d_t\|_{H^1}^2+\delta \|\nabla u_t\|_{L^2}^2.\\
\end{aligned}
\end{equation}
Substituting \eqref{4.36} into \eqref{4.34} and integrating the resulting inequality over $(0,t)$, it arrives at
\begin{equation*}
\begin{aligned}
&\quad \frac{1}{2}\int |\nabla^3 d|^2 dx+\int_0^t \int |\nabla^2 d_t|^2 dxd\tau\\
&\le C +C(\varepsilon, \delta)\int_0^t(1+\|\sqrt{\rho}u_t\|_{L^2}^2+\|\nabla^3 d\|_{L^2}^2)\|\nabla^3 d\|_{L^2}^2d\tau
       +\varepsilon \int_0^t\|\nabla^2 d_t\|_{L^2}^2d\tau \\
&\quad \quad +\delta \int_0^t \|\nabla u_t\|_{L^2}^2 d\tau
       +C\int(|\nabla u| |\nabla d|+|u| |\nabla^2 d|+|\nabla d| |\nabla^2 d|+|\nabla d|^3)|\nabla \Delta d|dx.
\end{aligned}
\end{equation*}
Choosing $\varepsilon=\frac{1}{2}$, we obtain
\begin{equation}\label{4.37}
\begin{aligned}
&\quad \int |\nabla^3 d|^2 dx+\int_0^t \int |\nabla^2 d_t|^2 dxd\tau\\
&\le C +C(\delta)\int_0^t(1+\|\sqrt{\rho}u_t\|_{L^2}^2+\|\nabla^3 d\|_{L^2}^2)\|\nabla^3 d\|_{L^2}^2d\tau
       +\delta \int_0^t \|\nabla u_t\|_{L^2}^2 d\tau\\
&\quad \quad  +C\int(|\nabla u| |\nabla d|+|u| |\nabla^2 d|+|\nabla d| |\nabla^2 d|+|\nabla d|^3)|\nabla \Delta d|dx\\
&=  C +C(\delta)\int_0^t(1+\|\sqrt{\rho}u_t\|_{L^2}^2+\|\nabla^3 d\|_{L^2}^2)\|\nabla^3 d\|_{L^2}^2d\tau
       +\delta \int_0^t \|\nabla u_t\|_{L^2}^2 d\tau+\sum_{i=1}^4 {I\!I}_{5i}.
\end{aligned}
\end{equation}
By using the H\"{o}lder, Gagliardo-Nirenberg and Young inequalities repeatedly, we get
\begin{equation*}
\begin{aligned}
{I\!I}_{51} &\le C\|\nabla d\|_{L^\infty} \|\nabla u\|_{L^2} \|\nabla \Delta d\|_{L^2}
        \le C\|\nabla d\|_{L^2}^\frac{1}{4} \|\nabla d\|_{H^2}^\frac{3}{4}\|\nabla u\|_{L^2} \|\nabla \Delta d\|_{L^2}\\
       &\le C(1+\|\nabla^3 d\|_{L^2})^\frac{3}{4} \|\nabla^3 d\|_{L^2}
        \le C(\varepsilon)+\varepsilon\|\nabla^3 d\|_{L^2}^2,\\
{I\!I}_{52} &\le C\|u\|_{L^6} \|\nabla^2 d\|_{L^3} \|\nabla \Delta d\|_{L^2}
        \le C\|\nabla u\|_{L^2} \|\nabla^2 d\|_{L^2}^\frac{1}{2} \|\nabla^2 d\|_{H^1}^\frac{1}{2}\|\nabla \Delta d\|_{L^2}\\
       &\le C(1+\|\nabla^3 d\|_{L^2})^\frac{1}{2} \|\nabla^3 d\|_{L^2}
        \le C(\varepsilon)+\varepsilon\|\nabla^3 d\|_{L^2}^2,\\
\end{aligned}
\end{equation*}
\begin{equation*}
\begin{aligned}
{I\!I}_{53} &\le C\|\nabla d\|_{L^6} \|\nabla^2 d\|_{L^3} \|\nabla \Delta d\|_{L^2}
        \le C\|\nabla d\|_{H^1} \|\nabla^2 d\|_{L^2}^\frac{1}{2} \|\nabla^2 d\|_{H^1}^\frac{1}{2}\|\nabla \Delta d\|_{L^2}\\
       &\le C(1+\|\nabla^3 d\|_{L^2})^\frac{1}{2} \|\nabla^3 d\|_{L^2}
        \le C(\varepsilon)+\varepsilon\|\nabla^3 d\|_{L^2}^2,\\
{I\!I}_{54} &\le C\|\nabla d\|_{L^6}^3  \|\nabla \Delta d\|_{L^2}
        \le C\|\nabla d\|_{H^1}^3 \|\nabla^3 d\|_{L^2}
        \le C(\varepsilon)+\varepsilon\|\nabla^3 d\|_{L^2}^2.
\end{aligned}
\end{equation*}
Substituting ${I\!I}_{5i}(1 \le i \le 4)$ into \eqref{4.37} and  choosing $\varepsilon$ small enough, it arrives at
\begin{equation}\label{4.38}
\begin{aligned}
&\quad \|\nabla^3 d\|_{L^2}^2+\int_0^t \|\nabla^2 d_t\|_{L^2}^2d\tau\\
&\le  C +C(\delta)\int_0^t(1+\|\sqrt{\rho}u_t\|_{L^2}^2+\|\nabla^3 d\|_{L^2}^2)\|\nabla^3 d\|_{L^2}^2d\tau
       +\delta \int_0^t \|\nabla u_t\|_{L^2}^2 d\tau.
\end{aligned}
\end{equation}
Adding \eqref{4.38} to \eqref{4.33} and choosing $\delta$ and $\eta$ suitably small, we obtain
\begin{equation*}
\begin{aligned}
&\quad (\|\sqrt{\rho}u_t\|_{L^2}^2+\|\nabla^3 d\|_{L^2}^2)+\int_0^t (\|\nabla u_t\|_{L^2}^2+\|\nabla^2 d_t\|_{L^2}^2) d\tau\\
&\le C+C\int_0^t(1+\|\sqrt{\rho}u_t\|_{L^2}^2+\|\nabla^3 d\|_{L^2}^2)(1+\|\nabla^3 d\|_{L^2}^2)d\tau,
\end{aligned}
\end{equation*}
which, together with Gr\"{o}nwall inequality and \eqref{4.7}, completes the proof.
\end{proof}

Finally, we derive the high order estimate for the strong solution $(\rho, u, P, d)$.

\begin{lemm}\label{lemma4.5}
Under the condition \eqref{4.1}, it holds that for $0 \le T<T^*$,
\begin{equation}\label{4.39}
\begin{aligned}
&\underset{0 \le t \le T}{\sup}(\|\rho_t\|_{L^q}+\|u\|_{H^2}+\|P\|_{H^1}+\|\nabla d_t\|_{L^2})\\
&\quad +\int_0^T \left(\|u\|_{W^{2,r}}^2+\|P\|_{W^{1,r}}^2+\|d_{tt}\|_{L^2}^2+\|\nabla^4 d\|_{L^2}^2\right)dt\le C.
\end{aligned}
\end{equation}
\end{lemm}
\begin{proof}
By \eqref{4.28} and \eqref{4.31}, it is easy to deduce
\begin{equation}\label{4.40}
\|u\|_{H^2}+\|P\|_{H^1}+\int_0^T \left(\|u\|_{W^{2,r}}^2+\|P\|_{W^{1,r}}^2\right)dt\le C,
\end{equation}
which, together with $\eqref{1.1}_1$, yields
\begin{equation}\label{4.41}
\|\rho_t\|_{L^q}\le \|u\|_{L^\infty}\|\nabla \rho\|_{L^q}\le C\|u\|_{H^2}\|\nabla \rho\|_{L^q} \le C.
\end{equation}
By \eqref{4.13}, \eqref{4.25}, \eqref{4.40} and Sobolev inequality, we obtain
\begin{equation}\label{4.42}
\begin{aligned}
&\quad \|\nabla d_t\|_{L^2}
=\|\nabla \Delta d+\nabla (|\nabla d|^2 d-u \cdot \nabla d)\|_{L^2}\\
& \le C(\|\nabla^3 d\|_{L^2}+\||\nabla u| |\nabla d|\|_{L^2}+\||u| |\nabla^2 d|\|_{L^2}
         +\|\nabla d\|_{L^6}^3+\||\nabla d| |\nabla^2 d|\|_{L^2})\\
& \le C(\|\nabla^3 d\|_{L^2}+ \|\nabla u\|_{L^2}\|\nabla d\|_{H^2}+\|u\|_{H^2}\|\nabla^2 d\|_{L^2}
        +\|\nabla d\|_{H^1}^3+\|\nabla d\|_{H^2}\|\nabla^2 d\|_{L^2})\\
&\le C.
\end{aligned}
\end{equation}
Differentiating $\eqref{1.1}_4$ with respect to $t$, it arrives at
\begin{equation}\label{4.43}
d_{tt}-\Delta d_t=(|\nabla d|^2 d-u \cdot \nabla d)_t.
\end{equation}
Taking $L^2$ estimate to \eqref{4.43}, by virtue of \eqref{4.24} and \eqref{4.40}, we obtain
\begin{equation}\label{4.44}
\begin{aligned}
&\quad \int_0^T \|d_{tt}\|_{L^2} dt\\
&\le C\int_0^T(\|\Delta d_t\|_{L^2}+\| |\nabla d_t| |\nabla d| \|_{L^2}+ \| |\nabla d|^2 |d_t|\|_{L^2}
               +\| |u_t| |\nabla d|\|_{L^2}+\| |u| |\nabla d_t|\|_{L^2})dt\\
& \le C\int_0^T\left(\|\nabla^2 d_t\|_{L^2}+\|\nabla d_t\|_{L^2}\|\nabla d\|_{H^2}+\|d_t\|_{L^2}\|\nabla d\|_{H^2}^2
               +\|\nabla u_t\|_{L^2}\|\nabla d\|_{H^1}\right.\\
& \quad \quad \quad \quad \left.+\|u\|_{H^2}\|\nabla d_t\|_{L^2}\right)dt\\
& \le C\int_0^T(\|\nabla^2 d_t\|_{L^2}+\|\nabla u_t\|_{L^2}+1)dt\le C.
\end{aligned}
\end{equation}
Taking $\nabla $ operator to $\eqref{1.1}_1$, we get
\begin{equation}\label{4.45}
-\nabla^2 \Delta d=\nabla^2(|\nabla d|^2 d- u\cdot \nabla d -d_t).
\end{equation}
Applying $H^4$ estimate to $d$ with boundary $\left.\frac{\partial d}{\partial \nu}\right|_{\partial \Omega}=0$ and \eqref{4.45}
we get
\begin{equation}\label{4.46}
\begin{aligned}
&\quad \|\nabla^4 d\|_{L^2} \le C(\|\nabla^2 \Delta d\|_{L^2}+\|\nabla^2 d\|_{H^1})\\
&\le C (\|\nabla^2(|\nabla d|^2 d- u\cdot \nabla d -d_t)\|_{L^2}+1)\\
&\le C \left(\|\nabla^2 d_t|\|_{L^2}+\||\nabla^2 d|^2\|_{L^2}+\||\nabla d|^2 |\nabla^2 d|\|_{L^2}
             +\| |\nabla^2 u| |\nabla d|\|_{L^2}\right.\\
&\quad \quad \left.  +\||\nabla u| |\nabla^2 d|\|_{L^2}+\| |u| |\nabla^3 d|\|_{L^2}\right)\\
&\le C\left(\|\nabla^2 d_t|\|_{L^2}+\|\nabla^2 d\|_{L^4}^2+\|\nabla d\|_{H^2}^2 \|\nabla^2 d\|_{L^2}
        +\|\nabla d\|_{H^2}\|\nabla^2 u\|_{L^2}\right.\\
&\quad \quad \left.+\|\nabla u\|_{H^1}\|\nabla^2 d\|_{H^1}+\|u\|_{H^2}\|\nabla^3 d\|_{L^2}\right)\\
&\le C(\|\nabla^2 d_t|\|_{L^2}+1),
\end{aligned}
\end{equation}
which means
\begin{equation*}
\int_0^T \|\nabla^4 d\|_{L^2}^2 dt \le C,
\end{equation*}
which completes the proof.
\end{proof}

After having the Lemmas \ref{lemma4.1}-\ref{lemma4.5} at hand, it is easy to apply the Theorem \ref{Theorem1.1} to extend
the strong solution $(\rho, u, P, d)$ beyond time $T^*$. Therefore, we complete the proof of Theorem \ref{blowupcriterion1}.

\section{Proof of Theorem \ref{blowupcriterion3}}
\quad In this section, we will also give the proof of Theorem \ref{blowupcriterion3} by contradiction. Assume $0 < T^* <\infty$ to be the maximum time for the existence of strong solution $(\rho, u, P, d)$ to \eqref{1.1}-\eqref{1.4}. Suppose that \eqref{1.10}
were false, that is
\begin{equation}\label{5.1}
M_1 \triangleq \underset{T \rightarrow T^*}{\lim}\left(\|\nabla \rho\|_{L^\infty(0,T;L^q)}
   +\|\nabla d\|_{L^{s}(0,T;L_w^{r})}\right)< \infty.
\end{equation}

\begin{lemm}\label{lemma5.1}
Under the condition \eqref{5.1}, it holds that for $0\le T< T^*$,
\begin{equation}\label{5.2}
\int_0^T \int |\nabla^2 d|^2 dxd\tau \le C,
\end{equation}
where $C$ denotes generic constants depending only on $\Omega, M_1,  T^*$ and the initial data.
\end{lemm}

\begin{proof}
By virtue of $|d|=1$, we have $d \cdot \Delta d=|\nabla d|^2$. Then, by \eqref{4.2}, we obtain
\begin{equation}\label{5.3}
\begin{aligned}
\int_0^T \int |\Delta d|^2 dxd\tau
&=\int_0^T \int |\nabla d|^4 dx d\tau+\int_0^T \int \left|\Delta d+|\nabla d|^2 d\right|^2 dxd\tau\\
&\le \int_0^T \int |\nabla d|^4 dx \tau+C.
\end{aligned}
\end{equation}
By Lemma \ref{lemma2.3} and \eqref{4.2} it is easy to deduce
\begin{equation*}
\begin{aligned}
\int |\nabla d|^4 dx
& =\int |\nabla d|^2 |\nabla d|^2 dx
  \le \varepsilon \|\nabla d\|_{H^1}^2
    +C(\varepsilon)\|\nabla d\|_{L_w^r}^\frac{2r}{r-2}\|\nabla d\|_{L^2}^2\\
&\le \varepsilon \|\nabla^2 d\|_{L^2}^2
    +C(\varepsilon)(1+\|\nabla d\|_{L_w^r}^\frac{2r}{r-2}).\\
\end{aligned}
\end{equation*}
Applying the elliptic regularity with Neumann boundary, it arrives at
\begin{equation*}
\begin{aligned}
\int_0^T \|\nabla^2 d\|_{L^2}^2d\tau
& \le C\int_0^T (\|\Delta d\|_{L^2}^2+\|d\|_{H^1}^2)d\tau\\
& \le \varepsilon\int_0^T \|\nabla^2 d\|_{L^2}^2 d\tau+C(\varepsilon)\int_0^T (1+\|\nabla d\|_{L_w^r}^\frac{2r}{r-2})d\tau\\
& \le  \varepsilon\int_0^T \|\nabla^2 d\|_{L^2}^2 d\tau+C(\varepsilon)\int_0^T (1+\|\nabla d\|_{L_w^r}^s)d\tau\\
& \le   \varepsilon\int_0^T \|\nabla^2 d\|_{L^2}^2 d\tau+C,
\end{aligned}
\end{equation*}
where, $r$ and $s$ satisfy the condition \eqref{1.11}. Choosing $\varepsilon=\frac{1}{2},$  we get
\begin{equation*}
\int_0^T \int |\nabla^2 d|^2 dxd\tau \le C,
\end{equation*}
which completes the proof.
\end{proof}

\begin{lemm}\label{lemma5.2}
Under the condition \eqref{5.1}, it holds that for $0 \le T <T^*$,
\begin{equation}\label{5.4}
\begin{aligned}
&\underset{0 \le t \le T}{\sup}\left(\|\nabla u\|_{L^2}^2+\|\nabla d\|_{L^4}^4+\|\nabla^2 d\|_{L^2}^2\right)\\
&+\int_0^T \int \left(|\sqrt{\rho}\dot{u}|^2+|\nabla d|^2 |\nabla^2 d|^2 +|\nabla^3 d|^2+|\nabla d_t|^2\right)dxd\tau < \infty.
\end{aligned}
\end{equation}
\end{lemm}

\begin{proof}
By \eqref{4.21}, it is easy to deduce, for any $0 \le s< t\le T,$
\begin{equation}\label{5.5}
\begin{aligned}
&\quad \int(|\nabla u|^2+|\nabla^2 d|^2+|\nabla d|^4)(t)dx
    +\int_s^t\int\left(\rho |\dot{u}|^2+|\nabla d|^2 |\nabla^2 d|^2+|\nabla(|\nabla d|^2)|^2\right.\\
&\quad \quad \left.+|\nabla^3 d|^2+|\nabla d_t|^2\right)dxd\tau\\
&\le C+C\int(|\nabla u|^2+|\nabla^2 d|^2+|\nabla d|^4)(s)dx+C\int_s^t (1+\|u\|_{L^\infty}^2+\|\nabla d\|_{L^\infty}^2) \\
&\quad \quad \times(1+\|\nabla u\|_{L^2}^2+\|\nabla^2 d\|_{L^2}^2+\|\nabla d\|_{L^4}^4)d\tau.
\end{aligned}
\end{equation}
Let
\begin{equation*}
\begin{aligned}
&A(t)=e+\int(|\nabla u|^2+|\nabla^2 d|^2+|\nabla d|^4)(t)dx
    +\int_0^t\int\left(\rho |\dot{u}|^2+|\nabla d|^2 |\nabla^2 d|^2+|\nabla(|\nabla d|^2)|^2\right.\\
&\quad \quad \quad \left. +|\nabla^3 d|^2+|\nabla d_t|^2\right)dxd\tau,
\end{aligned}
\end{equation*}
 then we deduce from \eqref{5.5} that
\begin{equation*}
A(t) \le CA(s)+C\int_0^t (1+\|u\|_{L^\infty}^2+\|\nabla d\|_{L^\infty}^2)A(\tau)d\tau,
\end{equation*}
which, together with the Gr\"{o}nwall inequality, gives directly that
\begin{equation*}
A(t) \le CA(s)\exp\left[C\int_s^t (1+\|u\|_{L^\infty}^2+\|\nabla d\|_{L^\infty}^2)d\tau\right].
\end{equation*}
Let
\begin{equation*}
\begin{aligned}
&\Phi(t)=e+\underset{0 \le \tau \le t}{\sup}(\|\nabla u\|_{L^2}^2+\|\nabla^2 d\|_{L^2}^2+\|\nabla d\|_{L^4}^4)
     +\int_0^t\int\left(\rho |\dot{u}|^2+|\nabla d|^2 |\nabla^2 d|^2+|\nabla(|\nabla d|^2)|^2\right.\\
&\quad \quad \quad \left.+|\nabla^3 d|^2+|\nabla d_t|^2\right)dxd\tau,
\end{aligned}
\end{equation*}
then we have
\begin{equation}\label{5.6}
\Phi(T) \le C\Phi(s)\exp\left[C\int_s^T (1+\|u\|_{L^\infty}^2+\|\nabla d\|_{L^\infty}^2)d\tau\right].
\end{equation}
Now we control the term $\int_s^T (1+\|u\|_{L^\infty}^2+\|\nabla d\|_{L^\infty}^2)d\tau$. Indeed, by the Lemma \ref{lemma2.4}, we have
\begin{equation}\label{5.7}
\begin{aligned}
& \quad C\int_s^T (1+\|u\|_{L^\infty}^2+\|\nabla d\|_{L^\infty}^2)d\tau\\
& \le C\left[(T-s)+\left(\|u\|_{L^2(s,T; H^1)}^2+\|\nabla d\|_{L^2(s,T; H^1)}^2\right)
        \left(\ln\left(e+\|u\|_{L^2(s,T; W^{1,3})}\right) \right.\right.\\
&\quad \quad \quad \left.\left.+\ln\left(e+\|\nabla d\|_{L^2(s,T; W^{1,3})}\right)\right)\right].
\end{aligned}
\end{equation}
Applying the Sobolev inequality and regularity estimate, it arrives at
\begin{equation}\label{5.8}
\|u\|_{W^{1,3}}^2 \le C\|u\|_{W^{2,2}}^2 \le C(\|\sqrt{\rho}\dot{u}\|_{L^2}^2+\||\nabla d| |\nabla^2 d|\|_{L^2}^2),
\end{equation}
\begin{equation}\label{5.9}
\|\nabla d\|_{W^{1,3}}^2 \le C\|\nabla d\|_{W^{2,2}}^2 \le C(1+\|\nabla^2 d\|_{L^2}^2+\|\nabla^3 d\|_{L^2}^2).
\end{equation}
Substituting \eqref{5.8}-\eqref{5.9} into \eqref{5.7} yields
\begin{equation*}
C\int_s^T (1+\|u\|_{L^\infty}^2+\|\nabla d\|_{L^\infty}^2)d\tau
\le C+\ln [C\Phi(T)]^{C\left(\|u\|_{L^2(s,T; H^1)}^2+\|\nabla d\|_{L^2(s,T; H^1)}^2\right)},
\end{equation*}
which, together with \eqref{5.6}, gives directly that
\begin{equation*}
\Phi(T) \le C\Phi(s)\Phi(T)^{C\left(\|u\|_{L^2(s,T; H^1)}^2+\|\nabla d\|_{L^2(s,T; H^1)}^2\right)}.
\end{equation*}
Thanks to \eqref{4.2} and \eqref{5.2}, we can choose $s$ closes enough to $T^*$ such that
\begin{equation*}
\underset{T\rightarrow T^*}{\lim}{C\left(\|u\|_{L^2(s,T; H^1)}^2+\|\nabla d\|_{L^2(s,T; H^1)}^2\right)} \le \frac{1}{2},
\end{equation*}
which, means
\begin{equation*}
\Phi(T)\le C\Phi(s)^2 <\infty.
\end{equation*}
 Thus, we complete the proof.
\end{proof}

\begin{rema}
Unfortunately, we cannot derive the bound, just depending on the initial data, for \eqref{5.4} owing to the technique used here. However, the bound is unform with respect to time in \eqref{5.4} since $s$, which closed enough to $T^*$, is fixed in process of the proof for Lemma \ref{lemma5.2}. Thus, we can rewrite \eqref{5.4} as
\begin{equation*}
\begin{aligned}
&\underset{t\in [0,T]}{\sup}\int \left(|\nabla u|^2+ |\nabla^2 d|^2+|\nabla d|^4\right)dx\\
&+\int_0^T \int \left(\rho |\dot{u}|^2+|\nabla d|^2|\nabla^2 d|^2+|\nabla d_t|^2+|\nabla^3 d|^2\right)dxdt\le C(s),
\end{aligned}
\end{equation*}
where and in what follows, $C(s)$ denotes generic constants depending not only on $\Omega, M_1, T^*$ and the initial data, but also on the data that is fixed on time $s$.
\end{rema}

Similar to Lemma \ref{lemma4.3}, as a corollary of Lemma \eqref{lemma4.2}, we have the following estimate.

\begin{coro}\label{lemma5.4}
Under the condition \eqref{5.1}, it holds that for $0 \le T<T^*$,
\begin{equation}\label{5.10}
\underset{0 \le t \le T}{\sup} \|d_t\|_{L^2}+\int_0^T \| u \|_{H^2}^2 d\tau \le C(s).
\end{equation}
\end{coro}

Now, we derive the high order estimate $\|\sqrt{\rho}u_t\|_{L^2}$ and $\|\nabla^3 d\|_{L^2}$.

\begin{lemm}\label{lemma5.5}
Under the condition \eqref{5.1}, it holds that for $0 \le T <T^*$,
\begin{equation}\label{5.11}
\underset{0 \le t \le T}{\sup}(\|\sqrt{\rho}u_t\|_{L^2}^2+\|\nabla^3 d\|_{L^2}^2)
+\int_0^t(\|\nabla u_t\|_{L^2}^2+ \|\nabla^2 d_t\|_{L^2}^2)d\tau
\le  C(s).
\end{equation}
\end{lemm}
\begin{proof}
Recall from \eqref{4.29}, we have the following estimate
\begin{equation}\label{5.12}
\begin{aligned}
&\quad \frac{1}{2}\frac{d}{dt}\int \rho |u_t|^2dx+\frac{1}{C}\int |\nabla u_t|^2 dx\\
&\le C(\varepsilon)(\|\sqrt{\rho}u_t\|_{L^2}^2+\|\nabla^3 d\|_{L^2}^2+1)
     +C(\varepsilon,\eta)\|\nabla d_t\|_{L^2}^2+\eta\|\nabla d_t\|_{H^1}^2+\varepsilon\|\nabla u_t\|_{L^2}^2\\
&\quad \quad +C\int |u| |\nabla \rho| |\nabla u| |\nabla u_t| dx.
\end{aligned}
\end{equation}
By \eqref{5.1}, H\"{o}lder, Sobolev ang Young inequalities, we obtain
\begin{equation}\label{5.13}
\begin{aligned}
\int |u| |\nabla \rho| |\nabla u| |\nabla u_t| dx
&\le \|u\|_{L^{\frac{2r}{r-2}}} \|\nabla u\|_{L^{\frac{rq}{q-r}}} \|\nabla \rho\|_{L^q} \|\nabla u_t\|_{L^2}\\
&\le C \|\nabla u\|_{L^2} \|\nabla u\|_{H^1} \|\nabla \rho\|_{L^q} \|\nabla u_t\|_{L^2}\\
&\le C(\varepsilon)  \|\nabla u\|_{H^1}^2+\varepsilon\|\nabla u_t\|_{L^2}^2\\
&\le C(\|\sqrt{\rho}u\|_{L^2}^2+\|\nabla^3 d\|_{L^2}^2+1)+\varepsilon\|\nabla u_t\|_{L^2}^2.\\
\end{aligned}
\end{equation}
Substituting \eqref{5.13} into \eqref{5.12}, we obtain
\begin{equation}\label{5.14}
\begin{aligned}
&\quad \frac{1}{2}\frac{d}{dt}\int \rho |u_t|^2dx+\frac{1}{C}\int |\nabla u_t|^2 dx\\
&\le C(\varepsilon)(\|\sqrt{\rho}u_t\|_{L^2}^2+\|\nabla^3 d\|_{L^2}^2+1)
     +C(\varepsilon,\eta)\|\nabla d_t\|_{L^2}^2+\eta\|\nabla d_t\|_{H^1}^2+\varepsilon\|\nabla u_t\|_{L^2}^2.
\end{aligned}
\end{equation}
Thanks to the compatibility condition, after choosing $\varepsilon$ small enough, we get
\begin{equation}\label{5.15}
\begin{aligned}
&\quad \|\sqrt{\rho}u_t\|_{L^2}^2+\int_0^t \|\nabla u_t\|_{L^2}^2 d\tau\\
&\le C(\eta)+C\int_0^t (1+\|\sqrt{\rho}u_t\|_{L^2}^2+\|\nabla^3 d\|_{L^2}^2) d\tau+\eta \int_0^t \|\nabla^2 d_t\|_{L^2}^2  d\tau.
\end{aligned}
\end{equation}
Recall from \eqref{4.38}, we have the estimate
\begin{equation}\label{5.16}
\begin{aligned}
&\quad \|\nabla^3 d\|_{L^2}^2+\int_0^t \|\nabla^2 d_t\|_{L^2}^2d\tau\\
&\le  C +C(\delta)\int_0^t(1+\|\sqrt{\rho}u_t\|_{L^2}^2+\|\nabla^3 d\|_{L^2}^2)\|\nabla^3 d\|_{L^2}^2d\tau
       +\delta \int_0^t \|\nabla u_t\|_{L^2}^2 d\tau.
\end{aligned}
\end{equation}
Adding \eqref{5.15} to \eqref{5.16} and choosing $\varepsilon$ and $\delta$ small enough, we have
\begin{equation*}
\begin{aligned}
&\quad (\|\sqrt{\rho}u_t\|_{L^2}^2+\|\nabla^3 d\|_{L^2}^2)
+\int_0^t(\|\nabla u_t\|_{L^2}^2+ \|\nabla^2 d_t\|_{L^2}^2)d\tau\\
&\le  C +C\int_0^t(1+\|\sqrt{\rho}u_t\|_{L^2}^2+\|\nabla^3 d\|_{L^2}^2)(1+\|\nabla^3 d\|_{L^2}^2)d\tau,
\end{aligned}
\end{equation*}
which, together with the Gr\"{o}nwall inequality, complete the proof of the lemma.
\end{proof}

As a corollary of Lemma \ref{lemma5.5}, it is easy to deduce the following estimate
\begin{equation*}
\begin{aligned}
&\underset{0 \le t \le T}{\sup}(\|\rho_t\|_{L^q}+\|u\|_{H^2}+\|P\|_{H^1}+\|\nabla d_t\|_{L^2})\\
&+\int_0^T \left(\|u\|_{W^{2,r}}^2+\|P\|_{W^{1,r}}^2+\|d_{tt}\|_{L^2}^2+\|\nabla^4 d\|_{L^2}^2\right)dt\le C(s).
\end{aligned}
\end{equation*}

Therefore, having all the estimates  at hand, it is easy to extend the strong solution beyond time $T^*$. Thus, we complete the proof of Theorem \ref{blowupcriterion3}.
\section{Proof of Corollary \ref{blowupcriterion5}}
\quad In this section, we will give the proof Corollary \ref{blowupcriterion5}. Indeed, let $(\rho, u, P, d)$ be the solution of \eqref{1.1}-\eqref{1.4}, we will derive some maximal principle for the direction field.

\begin{lemm}
For some given constant $\underline{d}_{0i}(i=1, 2)$, then we have the following maximal principle:\\
{\rm (i)}If $0 \le \underline{d}_{0i} \le d_{0i}\le 1$, then $0 \le \underline{d}_{0i} \le d_{i}\le 1$ for any $i=1,2;$\\
{\rm (ii)}If $-1\le d_{0i}\le -\underline{d}_{0i} \le 0 $, then $-1\le d_{i}\le -\underline{d}_{0i} \le 0$ for any $i=1,2.$\\
\end{lemm}
\begin{proof}
Since {\rm (i)} has been proved in \cite{Jiang-Jiang-Wang}, we only give the proof of {\rm (ii)}. Indeed, letting $V_i=d_i+\underline{d}_{0i}$ and $V^+_i=\max\{V_i, 0\}\ge 0,$ then we obtain
\begin{equation}\label{6.1}
\partial_t V_i-\Delta V_i=|\nabla d|^2 (V_i-\underline{d}_{0i})-u \cdot \nabla V_i.
\end{equation}
Multiplying \eqref{6.1} by $V_i^+$ and using the Neumann boundary condition
$\left.\frac{\partial d}{\partial \nu}\right|_{\partial \Omega}=0$, we get that
\begin{equation*}
\begin{aligned}
&\quad \frac{d}{dt}\frac{1}{2}\|V_i^+\|_{L^2}^2+\|\nabla V_i^+\|_{L^2}^2
  =\int \left[|\nabla d|^2 (V_i-\underline{d}_{0i})-u \cdot \nabla V_i\right]\cdot V_i^+ dx\\
&=\int |\nabla d|^2 |V_i^+|^2 dx-\int |\nabla d|^2 \underline{d}_{0i}V^+_i dx +\frac{1}{2}\int {\rm div}u |V_i^+|^2 dx\\
&\le\int |\nabla d|^2 |V_i^+|^2 dx \le \|\nabla d\|_{L^\infty}^2 \|V_i^+\|_{L^2}^2.
\end{aligned}
\end{equation*}
Hence, if we applying Gr\"{o}nwall inequality to the above inequality, we get
\begin{equation*}
\|V_i^+(t)\|_{L^2}^2 \le \|V_i^+(0)\|_{L^2}^2 \exp\left(C\int_0^t \|\nabla d\|_{L^\infty}^2 d\tau\right)=0,
\end{equation*}
which implies $-1 \le d_i \le -\underline{d}_{0i}\le 0.$
\end{proof}

Before starting our main result in this section, we recall the elliptic estimate
\begin{equation}\label{6.2}
\|\nabla^2 f\|_{L^2}^2 \le C_1 (\|\Delta f\|_{L^2}^2+\|f\|_{H^1}^2),
\end{equation}
and the Gagliardo-Nirenberg inequality
\begin{equation}\label{6.3}
\|\nabla f\|_{L^4}^4 \le C_2 (\|\nabla^2 f\|_{L^2}^2\|f\|_{L^\infty}^2+\|f\|_{L^\infty}^4),
\end{equation}
where the constants $C_1$ and $C_2$ independent of the function $f$.

\begin{lemm}
For any $i(i=1,2),$ if the $i-$th component of initial direction field $d_{0i}$ satisfies the condition
\begin{equation}\label{6.4}
0 \le \underline{d}_{0i} \le d_{0i} \le 1  \quad {\rm or} \quad -1\le d_{0i} \le -\underline{d}_{0i}\le 0,
\end{equation}
where $\underline{d}_{0i}$ is a constant and is defined by
\begin{equation}\label{6.5}
\underline{d}_{0i}>1-\frac{1}{2C_1 C_2},
\end{equation}
where the constant $C_1$ and $C_2$ are defined in \eqref{6.2} and \eqref{6.3} respectively.
Then we have the following estimate
\begin{equation}\label{6.6}
\int_0^T \|\nabla d\|_{L^4}^4 dt \le C.
\end{equation}
\end{lemm}
\begin{proof}We only give the proof for the case $i=2$.\\
Indeed, if $0 \le \underline{d}_{02} \le d_{02} \le 1$, then applying the maximal principle, we obtain
\begin{equation}
0 \le \underline{d}_{02} \le d_{2} \le 1 \quad {\rm for \ any  } \ t>0 \ {\rm  and \ any} \  x \in \Omega.
\end{equation}
By virtue of $|d|=1,$ then
\begin{equation}\label{6.8}
\|d-e_2\|_{L^\infty}^2 \le 2(1-\underline{d}_{02}),
\end{equation}
where $e_2=(0,1).$ Substituting \eqref{6.4}-\eqref{6.5} into  \eqref{6.8}, we get
\begin{equation}\label{6.9}
C_1 C_2 \|d-e_2\|_{L^\infty}^2<1.
\end{equation}
Combining \eqref{6.2} with \eqref{6.3}, we get
\begin{equation}\label{6.10}
\begin{aligned}
&\quad \|\nabla^2 d\|_{L^2}^2 \le C_1 (\|\Delta d\|_{L^2}^2+\|d\|_{H^1}^2)\\
&\le C_1 (\|\nabla d\|_{L^4}^4+\| \Delta d+|\nabla d|^2 d\|_{L^2}^2+\|d\|_{H^1}^2)\\
&\le C_1 C_2 \|d-e_2\|_{L^\infty}^2\|\nabla^2 d\|_{L^2}^2
     +C_1 \left(C_2\|d-e_2\|_{L^\infty}^4+\|\Delta d+|\nabla d|^2 d\|_{L^2}^2+\|d\|_{H^1}^2\right),
\end{aligned}
\end{equation}
which, together with \eqref{6.8}, gives
\begin{equation}\label{6.11}
\|\nabla^2 d\|_{L^2}^2 \le C\left(\|d-e_2\|_{L^\infty}^4+\|\Delta d+|\nabla d|^2 d\|_{L^2}^2+\|d\|_{H^1}^2\right).
\end{equation}
In view of the basic energy inequality \eqref{4.2}, we have
\begin{equation*}
\int_0^T \|\nabla^2 d\|_{L^2}^2 dt \le C,
\end{equation*}
which, together with \eqref{6.3}, implies
\begin{equation}\label{6.12}
\int_0^T \|\nabla d\|_{L^4}^4 dt \le C.
\end{equation}
As for the case $-1\le d_{0i} \le -\underline{d}_{0i}\le 0$, it is simple fact if we replaced the function
$d-e_2$ by $d+e_2$ in \eqref{6.8}-\eqref{6.11} respectively. Hence, we complete the proof.
\end{proof}

Choosing $r=s=4$ in \eqref{1.11}, if the maximal existence of time $T^*<\infty$, we have
\begin{equation*}
\underset{T \rightarrow T^*}{\lim}\left(\|\nabla \rho\|_{L^\infty(0,T;L^q)}+\|\nabla d\|_{L^{4}(0,T;L_w^{4})}\right)=\infty,
\end{equation*}
which, together with \eqref{6.12}, gives immediately
\begin{equation*}
\underset{T \rightarrow T^*}{\lim}\|\nabla \rho\|_{L^\infty(0,T;L^q)}=\infty.
\end{equation*}
Therefore, we complete the proof of Corollary \ref{blowupcriterion5}.\\

\noindent\textbf{Acknowledgements}\quad
Qiang Tao's research was supported by the NSF(Grant No.11171060) and the NSF(Grant No.11301345).
Zheng-an Yao's research was supported in part by NNSFC(Grant No.11271381) and China 973 Program(Grant No. 2011CB808002).

\phantomsection

\end{document}